\documentclass[a4paper,12pt]{amsart}
\usepackage{amssymb}
\usepackage{amsmath}
\usepackage{stmaryrd}
\usepackage{amscd,amsthm,amssymb}
\usepackage{enumerate}
\usepackage{color}

\scrollmode
\usepackage{latexsym}

\def\.{\cdot}

\def\d{{\delta}}

\def\vs{\vskip .6cm}

\def\beq{\begin{equation}}
\def\eeq{\end{equation}}
\def\bea{\begin{eqnarray*}}
\def\eea{\end{eqnarray*}}
\def\beaa{\begin{eqnarray}}
\def\eeaa{\end{eqnarray}}
\def\ba{\begin{array}}
\def\ea{\end{array}}
\def\f{\varphi}

\def\L{\Lambda}

\def \RM{\mathbb{R}}

\def \ZM{\mathbb{Z}}

\def \SM{\mathbb{S}}

\def\Ric{\mathrm{Ric}}
\def\id{\mathrm{id}}
\def\be{\begin{equation}}
\def\ee{\end{equation}}
\def\tr{\mathrm{tr}}
\def\Aut{\mathrm{Aut }}

\def\Hom{\mathrm{Hom}}
\def\Sym{\mathrm{Sym}}

\def\so{\mathfrak{so}}

\def\SO{\mathrm{SO}}
\def\OO{\mathrm{O}}
\def\End{\mathrm{End}}

\def\Sym{\mathrm{Sym}}
\def\scal{\mathrm{scal}}

\def\Id{\mathrm{id}}

\def\T{\mathrm{\,T}}

\def\Rh{\mathring{R}}
\def\lr{\,\lrcorner\,}
\def\Cas{\mathrm{Cas} }
\def\dd{\mathrm{d}}
\def\LL{\mathrm{L}}
\def\grad{\mathrm{grad}}

\newtheorem{epr}{Proposition}[section]
\newtheorem{ath}[epr]{Theorem}
\newtheorem{elem}[epr]{Lemma}
\newtheorem{ecor}[epr]{Corollary}

\theoremstyle{definition}
\newtheorem{ede}[epr]{Definition}
\newtheorem{ere}[epr]{Remark}
\newtheorem{exe}[epr]{Example}

\title[Conformal Killing tensors]{Killing and Conformal Killing tensors}

\address{Konstantin Heil\\
Institut f\"ur Geometrie und Topologie \\
Fachbereich Mathematik\\
Universit{\"a}t Stuttgart\\
Pfaffenwaldring 57 \\
70569 Stuttgart, Germany
}
\email{konstantin.heil@mathematik.uni-stuttgart.de}

\author{Konstantin Heil, Andrei Moroianu, Uwe Semmelmann}

\address{Andrei Moroianu \\ Laboratoire de Math\'ematiques de Versailles, UVSQ, CNRS, Universit\'e Paris-Saclay, 78035 Versailles, France }
\email{andrei.moroianu@math.cnrs.fr}

\address{Uwe Semmelmann\\
Institut f\"ur Geometrie und Topologie \\
Fachbereich Mathematik\\
Universit{\"a}t Stuttgart\\
Pfaffenwaldring 57 \\
70569 Stuttgart, Germany
}
\email{uwe.semmelmann@mathematik.uni-stuttgart.de}

\date{\today}

\begin{document}

\begin{abstract}
We introduce an appropriate formalism in order to study conformal Killing (symmetric) tensors on Riemannian manifolds. 
We reprove in a simple way some known results in the field and obtain several new results, like the classification of 
conformal Killing $2$-tensors on Riemannian products of compact manifolds, Weitzenb\"ock formulas leading to non-existence 
results, and construct various examples of manifolds with conformal Killing tensors.
\vs
\noindent 2010 {\it Mathematics Subject Classification}: Primary: {53C25, 53C27, 53C40}
\smallskip

\noindent {\it Keywords}: {Killing tensors, conformal Killing tensors, St\"ackel tensors.} 
\end{abstract}

\maketitle

\section{Introduction}

Killing $p$-tensors are symmetric $p$-tensors with vanishing symmetrized covariant derivative. This is a natural generalization
of the Killing vector field equation for $p=1$. Since many years Killing tensors, and more generally conformal Killing tensors,  were intensively studied in the physics 
literature, e.g. in  \cite{penrose} and \cite{wood}. The original motivation came from the fact that symmetric Killing tensors define (polynomial) first integrals of the equations of motion, i.e. functions which are constant on geodesics. Conformal Killing tensors still define first integrals for null geodesics.
Killing  $2$-tensors also appeared in the analysis of the stability of generalized black holes in 
$D=11$ supergravity, e.g. in \cite{gibbons03}  and \cite{page}. It turns out that trace-free Killing $2$-tensors (also called St\"ackel tensors) precisely correspond to the limiting case
of a lower bound for the spectrum of the Lichnerowicz Laplacian on symmetric $2$-tensors. More recently,
Killing and conformal Killing tensors appeared in several other areas of mathematics, e.g.
in connection with geometric inverse problems, integrable systems and Einstein-Weyl geometry, cf. \cite{dairbekov},  \cite{fox1}, \cite{fox2}, \cite{salo2}, \cite{jelonek99}, \cite{salo1}, \cite{schoebel1}.

Any parallel tensor is in particular a Killing tensor.
The simplest non-parallel examples of Killing tensors can be constructed as symmetric products of Killing vector fields. For the standard sphere $\SM^n$
there is a direct correspondence between Killing tensors and algebraic curvature tensors on $\RM^{n+1}$.
Other interesting examples are obtained as Ricci tensors of certain Riemannian manifolds, e.g. of natural reductive spaces.

The defining equation of trace-free conformal Killing tensors has the important property to be of finite type (or strongly elliptic). This leads to an explicit upper
bound of the dimension of the space of conformal Killing tensors. In this respect, conformal Killing tensors are very similar to so-called conformal Killing forms,
which were studied by the authors in several articles, e.g. \cite{andrei15}, \cite{andreiuwe1} and  \cite{uwe}.  Moreover there is an explicit construction of Killing tensors starting from Killing forms, cf. Section \ref{4.3} below.

The existing literature on symmetric Killing tensors is huge, especially coming from theoretical physics. One of the main obstacles in reading it is the old-fashioned formalism used in most articles in the subject.

In this article we introduce conformal Killing tensors in a modern, coordinate-free formalism. We use this formalism in order to reprove in a simpler way some known results, like Theorem \ref{nodal} saying that the nodal set of a conformal Killing tensor has at least codimension 2, or Proposition \ref{ptf} showing the non-existence of trace-free conformal Killing tensors on compact manifolds of negative sectional curvature. In addition we give a unified treatment of some subclasses of conformal Killing tensors, e.g. special conformal Killing tensors.

We obtain several new results, like the classification of St\"ackel 2-tensors with at most two eigenvalues, or the description of conformal Killing 2-tensors on Riemannian products of compact manifolds (which turn out to be determined by Killing 2-tensors and Killing vector fields on the factors, cf. Theorem \ref{product}). 
We also prove a general Weitzenb\"ock formula (Proposition \ref{wbf}) leading to non-existence results on certain compact manifolds.

{\sc Acknowledgments.} This work was supported by the Procope Project No. 32977YJ. We are grateful to Mikko Salo who discovered an error in the proof of Proposition \ref{ptf}, and to Gregor Weingart who helped us to correct this error. We also thank the anonymous referee for having pointed out an error in the previous version of Theorem \ref{product} and for several useful suggestions.

\section{Preliminaries}

Let $(V, g)$ be a Euclidean vector space of dimension $n$. We denote with
$\Sym^p V\subset V^{\otimes p}$ the $p$-fold symmetric tensor product of $V$. Elements of $\Sym^p V$
are symmetrized tensor products
$$
v_1 \cdot \ldots \cdot   v_p := \sum_{\sigma \in S_p} \, v_{\sigma(1)} \otimes \ldots \otimes v_{\sigma(p)} \ ,
$$
where $v_1, \ldots, v_p$ are vectors in $V$. In particular we have
$v\cdot u = v\otimes u + u \otimes v$ for $u, v \in V$. 
Using the metric $g$, one can identify $V$ with $V^*$. Under this identification, $g\in\Sym^2V^*\simeq\Sym^2V$ can be written as $g=\frac12  \sum e_i \cdot e_i$, for any
orthonormal basis $\{ e_i\}$.

The direct sum $\Sym V:=\bigoplus_{p\ge 0} \Sym^p V$ is endowed with a commutative product making $\Sym V$ into a $\ZM$-graded commutative algebra.
The scalar product $g$ induces a scalar product, also denoted by $g$, on $\Sym^p V$ defined by 
$$
g(v_1 \cdot \ldots \cdot v_p,w_1 \cdot \ldots \cdot w_p) \;=\; \sum_{\sigma \in S_p}
\, g(v_1, w_{\sigma(1)}) \cdot \ldots  \cdot g(v_p, w_{\sigma(p)})\ .
$$
With respect to this scalar product, every element $K$ of $\Sym^p V$ can be identified with a symmetric $p$-linear map (i.e. a polynomial of degree $p$) on $V$ by the formula 
$$K(v_1,\ldots,v_p)=g(K,v_1\cdot\ldots\cdot v_p)\ .
$$
For every $v\in V$, the metric adjoint of the linear map
$
v\cdot: \Sym^p V \rightarrow \Sym^{p+1} V, \;  K \mapsto v \cdot K
$
is the contraction
$
v\lrcorner : \Sym^{p+1} V \rightarrow \Sym^{p} V, \;  K \mapsto v \lrcorner \, K
$, defined by $(v \lrcorner \, K) (v_1, \ldots , v_{p-1}) = K(v, v_1, \ldots, v_{p-1})$.
In particular we have $v \lrcorner \, u^p = p g(v, u) u ^{p-1}, \forall\  v, u \in V$.

We introduce the linear map $\, \deg : \Sym V \rightarrow \Sym V$, defined by
$\deg (K)  = p \,K $ for $ K \in \Sym^p V$. Then we have\;
$
\sum e_i \cdot e_i \lrcorner \, K = \deg(K)
$,
where $\{e_i\}$ as usual denotes an orthonormal frame. Note that if $K\in \Sym^p T$
is considered as a polynomial of degree $p$ then $v \lrcorner K$ corresponds to
the directional derivative $\partial_v K$ and the last formula is nothing else than
the well-known Euler theorem on homogeneous functions.

Contraction and multiplication with the metric $g$ defines two  additional  linear maps:
$$
\L : \Sym^p V \rightarrow \Sym^{p-2} V , \quad K \mapsto  \sum e_i \lrcorner \,  e_i \lrcorner \,  K 
$$
and
$$
\LL  : \Sym^{p-2} V \rightarrow \Sym^{p} V , \quad K \mapsto  \sum e_i \cdot e_i \cdot K \ ,
$$
which are adjoint to each other.
Note that $\LL (1) = 2 g$ and $\Lambda K = \tr(K)$ for every $K\in\Sym^2 V$.
It is straightforward to check the following algebraic commutator relations
\beq\label{commu}
[\, \Lambda, \, \LL  \,] \;=\; 2n\,  \Id \;+\;4 \deg,\quad [\, \deg, \LL  \,] = 2\, \LL , \quad [\,\deg, \L\,] = -\,2 \,\Lambda  \ ,
\eeq
and for every $v\in V$:
\beq \label{commu2}
[\, \Lambda, \, v \,\cdot \, ] \;=\; 2\, v \,  \lrcorner  \, , \quad
[\, v \lrcorner \,,\, \LL  \,  \,] \;=\; 2\, v \cdot\, ,\quad
[\,\L , \, v \lrcorner \, \,] \;=\; 0 \;=\; [\, \LL , \, v \cdot \,] \ . 
\eeq

For $V = \RM^n$, the standard $\OO(n)$-representation induces a reducible
$\OO(n)$-repre\-sentation on $\Sym^p V$. We denote by
$
\Sym^p_0 V := \ker( \L : \Sym^p V \rightarrow \Sym^{p-2} V)
$
the space of trace-free symmetric $p$-tensors. 

It is well known that $\Sym^p_0 \RM^n$ is an irreducible $\OO(n)$-representation and
we have the following decomposition into irreducible summands
$$
\Sym^p V  \;\cong\;  \Sym^p_0 V \;\oplus\; \Sym^{p-2}_0 V \;\oplus\;  \ldots \ ,
$$
where the last summand
in the decomposition is $\RM$ for $p$ even and $V$ for $p$ odd.
The summands 
$\Sym^{p- 2i}_0 V$ are embedded into $\Sym^pV$ via the map $\LL ^i$.  Corresponding to the
decomposition above any $K \in \Sym^p V$ can be decomposed as
$$
K \;=\;  K_0 \;+\;  \LL K_1 \;+\;  \LL ^2K_2 \;+ \; \ldots
$$
with $K_i \in \Sym^{p-2i}_0 V$, i.e. $\Lambda K_i = 0$. 
We will call this decomposition the {\it standard decomposition} of $K$.
In the following, the
subscript $0$ always  denotes the projection of an element from $\Sym^p V$ onto its
component in $\Sym^p_0 V$. Note that for any $v \in V$ and $K \in \Sym^p_0 V$ 
we have the following projection formula
\beq \label{projection}
(v \cdot K )_0 \;=\; v \cdot K \;-\; \tfrac{1}{n + 2 (p-1)}\, \LL  \,  (v \lrcorner \, K)\ .
\eeq
Indeed, using the commutator relation \eqref{commu} we have
$
\L ( \LL  \,  (v \lrcorner \,  K)) = (2n + 4(p-1)) \,  (v \lrcorner \,  K) 
$,
since $\Lambda $ commutes  with $v \lrcorner \, $ and $\Lambda K=0$. Moreover $\Lambda (v \cdot K) = 2\, v \lrcorner \, K$. Thus the right-hand side of 
\eqref{projection} is in the kernel of $\L$, i.e. it computes the projection $(v \cdot K)_0$. 

Recall the classical decomposition into irreducible  $\OO(n)$ representations 
\beq\label{deco}
V \otimes \Sym^p_0 V \;\cong\; \Sym^{p+1}_0V \;\oplus\; \Sym^{p-1}_0V \;\oplus\; \Sym^{p,1} V \ ,
\eeq
where $V = \RM^n$ is the standard $\OO(n)$-representation of highest weight
$(1,0,\ldots , 0)$, $\Sym^p_0 V$  
is the irreducible representation of highest weight $(p, 0, \ldots, 0)$ and 
$\Sym^{p,1} V$ is the irreducible representation of highest weight
$(p,1,0, \ldots, 0)$. We note that  $\Sym^{p+1}_0V$ is the so-called Cartan summand.
Its highest weight is the sum of the highest weights of $V$ and $\Sym^p_0V$.

Next we want to describe projections and embeddings for the first two summands.
The projection 
$
\pi_1 : V \otimes \Sym^p_0V \rightarrow \Sym^{p+1}_0 V
$
onto the first summand is defined as
\beq\label{pi1}
\pi_1(v \otimes K) \;:=\; (v \cdot K)_0 \;\stackrel{ \eqref{projection}}{=}\; v \cdot K \;-\; \tfrac{1}{n + 2(p - 1)}\, \LL \,  (v \lrcorner \,  K) \ .
\eeq
The adjoint map
$
\pi_1^* : \Sym^{p+1}_0 V \rightarrow V \otimes \Sym^p_0 V
$
is easily computed to be 
$
\pi^*_1(K) = \sum e_i \otimes (e_i \lrcorner \, K)
$.
Note that for any vector $v\in V$ the symmetric tensor
$v\lrcorner \, K$ is again trace-free, because $\, v \lrcorner \,$ commutes with $\L$.
Since $\pi_1 \, \pi^*_1 = (p+1) \, \Id$ on $\Sym^{p+1}_0V$, we conclude that
\begin{equation}\label{pi11}
p_1 \;:=\; \tfrac{1}{p+1} \, \pi^*_1 \, \pi_1 : V \otimes \Sym^p_0 V \;\rightarrow \;
\Sym^{p+1}_0 V \;\subset \; V \otimes \Sym^p_0 V
\end{equation}
 is the projection onto the irreducible summand of $V \otimes \Sym^p_0 V$ isomorphic to
 $\Sym^{p+1}_0V $.

Similarly the projection 
$
\pi_2 : V \otimes \Sym^p_0 V \rightarrow \Sym^{p-1}_0 V
$
onto the second summand in the decomposition \eqref{deco} is given by the contraction map
$
\pi_2 ( v \otimes K)  := v \lrcorner \, K
$.
In this case the adjoint map 
$
\pi_2^*: \Sym^{p-1}_0 V \rightarrow V \otimes \Sym^p_0 V
$
is computed to be
$$
\pi^*_2(K) \;=\; \sum e_i \otimes (e_i \cdot K)_0
\;=\; 
\sum e_i \otimes (e_i \cdot P - \tfrac{1}{n+2(p-2)} \, \LL  \, (e_i \lrcorner \, P)) \ . 
$$
It follows that
$\;
\pi_2 \, \pi^*_2 = (n+p-1) \, \id \;-\; \tfrac{2(p-2)}{n + 2(p-2)} \, \id
=
\tfrac{(n+2p-2)(n+p-3)}{n+2p-4}\,\id \ .
$
Thus the projection onto the irreducible summand in
$V \otimes \Sym^p_0 V$ isomorphic to $\Sym^{p-1}_0 V$  is given by
\begin{equation}\label{pi2}
p_2 \;:=\; \tfrac{n+2p-4}{(n+2p-2)(n+p-3)} \, \pi_2^* \, \pi_2 
: V \otimes \Sym^p_0 V \;\rightarrow \; \Sym^{p-1}_0 V \;\subset \; V \otimes \Sym^p_0 V \ .
\end{equation}
The projection $p_3$ onto the third irreducible summand in  $V \otimes \Sym^p_0 V$ is obviously given by $p_3 = \id - p_1 - p_2$ .

\medskip

Let $(M^n, g)$ be a Riemannian manifold with Levi-Civita connection $\nabla$. All the algebraic considerations above extend
to vector bundles over $M$, e.g. the $\OO(n)$-representation $\Sym^p V$
defines the real vector bundle $\Sym^p \T M$. The $\OO(n)$-equivariant maps
$\LL $ and $\L$ define bundle maps between the corresponding bundles. The same
is true for the symmetric product and the contraction $\iota$, as well as for the
maps $\pi_1, \pi_2$ and their adjoints, and the projection maps $p_1, p_2, p_3$. We
will use the same notation for the bundle maps on $M$.

Next we will define first order differential operators on sections of $\Sym^p \T M$. We have
$$
\dd : \Gamma(\Sym^p \T M) \rightarrow \Gamma( \Sym^{p+1}\T M), \quad K \mapsto \sum e_i \cdot \nabla_{e_i}K \ ,
$$
where $\{e_i\}$ denotes from now on a local orthonormal frame. The formal adjoint of 
$\dd$ is the divergence operator  $\d$ defined by
$$
\d : \Gamma(\Sym^{p+1} \T M) \rightarrow \Gamma(\Sym^{p}\T M), \quad K \mapsto -  \sum e_i \lrcorner\, \nabla_{e_i}K \ ,
$$
An immediate consequence of  the definition is 
\begin{elem} \label{derivation}
The operator $\dd$ acts as a derivation on the algebra of symmetric tensors, i.e.
for any $A \in \Gamma(\Sym^p \T M)$ and $B \in  \Gamma(\Sym^q \T M)$
the following equation holds
$$
\dd ( A \cdot B) \;=\; (\dd A) \cdot B \,+\, A \cdot (\dd B) \ .
$$
\end{elem}

An easy calculation proves that the operators $\dd$ and $\d$ satisfy the  commutator relations:
\beq\label{commu3}
[ \, \Lambda, \,\d\,]  \;=\; 0 \;=\;  [\,\LL ,\, \dd\, ], \quad [\, \Lambda, \,\dd\,] \; =\; -2 \d, \quad [\, \LL ,\, \d\, ] \;=\;  2 \,\dd\ .
\eeq

\medskip

\begin{elem}\label{constants}
Let  $K = K_0 +\LL   K_1 + \ldots$ be the standard decomposition of a section of  $ \Sym^p\T M $, where $K_i \in \Sym^{p-2i}_0\T M$. Then there exist real constants $a_i$ such that
$$
\dd K_i - a_i \LL  \d K_i \in \Sym^{p-2i+1}_0\T M \ .
$$
The constants are given explicitly  by  $a_i : = -\frac{1}{n+2(p-2i-1)} $. In particular, if  $K$ is  a section of $\Sym^p_0 \T M$, it holds that
\beq\label{dprojection}
(\dd K)_0 \;=\; \dd K \;+\; \tfrac{1}{n+2p-2}\,\LL  \, \d K \ .
\eeq
\end{elem}
\proof 
We write $K = \sum_{i\ge 0} \, \LL ^i \, K_i$, where $K_i$ is a section of  $\Sym^{p-2i}_0\T M$.
Then $\dd K_i - a_i \LL  \d K_i $ is a section of  $\Sym^{p-2i+1}_0\T M$ if and only if
$$
0 \;=\; \Lambda (\dd K_i - a_i \LL  \d K_i ) \;=\; - 2 \,\delta \, K_i \,-\, a_i\,(2n + 4(p-2i -1))\, \delta \, K_i \ .
$$
Thus the constants $a_i$ are as stated above. In particular we have  for $i=0$ that the expression $\dd K_0 + \frac{1}{n+2p-2}\LL \delta K_0$ is
trace-free. This proves the last statement.
\qed

\medskip

The operators $\dd$ and $\d$ can be considered as components of the covariant
derivative $\nabla$ acting on sections of $\Sym^p \T M$. To make this more precise
we first note that
\be\label{pi12}
\pi_1(\nabla K) \;=\; (\dd K)_0
\qquad \mbox{and} \qquad
\pi_2(\nabla K) \;=\; -\, \d K \ ,
\ee
which follows from $\nabla K = \sum e_i \otimes \nabla_{e_i} K$ and the definitions above.

\medskip

Let $K$ be a section of $\Sym^p_0\T M$. Then $\nabla K$ is a  section of $\T M \otimes \Sym^p_0 \T M$
and we may decompose $\nabla K$ corresponding to \eqref{deco}, i.e.
$
\nabla K = P_1(K)  + P_2(K) + P_3(K)
$,
where we use the notation $P_i(K) := p_i(\nabla K), i=1,2,3$.
Substituting the definition of the operators $P_i$ and applying the resulting equation to a tangent vector $X$
we obtain
\bea
\nabla_X K &=& \tfrac{1}{p+1}\, \pi_1^* (\dd K)_0(X) \;-\;  \tfrac{n+2p-4}{(n+2p-2)(n+p-3)} \,\pi_2^*(\d K)(X) \;+\;P_3(K) (X) \\[1ex]
&=&
\tfrac{1}{p+1} \, X \lrcorner \, (\dd K)_0 \;-\; \tfrac{n+2p-4}{(n+2p-2)(n+p-3)} \, (X\cdot \d K)_0
\;+\; P_3(K)(X)
\eea

Using \eqref{projection} and \eqref{dprojection}
we rewrite the formula for $\nabla_XK$ in terms of $\dd K$ and $\d K$ and obtain
\bea
\nabla_XK &=&
\tfrac{1}{p+1}\, X \lrcorner \, \dd K \;+\; 
\left(  \tfrac{1}{(n+2p-2)(p+1)} +\tfrac{1}{(n+2p-2)(n+p-3)}  \right) \, \LL \, (X \lrcorner \, \d K)
\\[.5ex]
&&
\phantom{xxxxxxxxxxxxxxx}
\;+\; 
\left(  \tfrac{2}{(n+2p-2)(p+1)} - \tfrac{n+2p-4}{(n+2p-2)(n+p-3)} \right)\, X \cdot \d  K
\;+\; P_3(K)(X)\\[1.5ex]
&=&
\tfrac{1}{p+1}\, X \lrcorner \, \dd K \;+\; 
\tfrac{1}{(p+1)(n+p-3)}\, \LL \, (X \lrcorner \, \d K)
\;-\;
\tfrac{p-1}{(p+1)(n+p-3)}\,X \cdot \d  K
\;+\; P_3(K)(X) \ .\\[1ex]
\eea
Here we  applied the commutator formula 
$X \lrcorner \, \LL  K = \LL  (X \lrcorner \, K) + 2 X \cdot K$. For later use we still note the
formulas
$$
P_1(K)(X) \;=\; \tfrac{1}{p+1} \, X \lrcorner \, (\dd K)_0 
\qquad \mbox{and}\qquad
P_2(K)(X) \;=\; -\,\tfrac{n+2p-4}{(n+2p-2)(n+p-3)} \, (X\cdot \d K)_0 \ .
$$

\medskip

At the end of this section we want to clarify the relations between $\dd , \d$ and $P_1, P_2$.
For convenience we introduce the notation $\dd_0K := (\dd K)_0$.  The relation between
$\dd_0$ and $\dd $ is given in \eqref{dprojection}. An easy calculation shows that $\d^\ast \;=\; \dd_0\ .$

\begin{elem}
On sections of $\Sym^p_0 \T M$ the following equations hold:
$$
\dd_0^\ast \, \dd^{\phantom \ast}_0 \;=\; (p+1)\, P_1^*  P^{\phantom *}_1
\qquad \mbox{and} \qquad
 \d^* \, \d \;=\;\tfrac{(n+2p-2)(n+p-3)}{n+2p-4} \, P^*_2 P^{\phantom *}_2 \ .
 $$ 
\end{elem}
\proof
Let $EM$ be a vector bundle associated to the frame bundle via a $\SO(n)$
representation E. The Levi-Civita connection induces a covariant derivative $\nabla$ acting on sections of $EM$.
If $\T $ denotes the tangent representation, defining the tangent bundle, we
have a decomposition into irreducible summands:
$
E \otimes \T  = E_1 \oplus \ldots \oplus E_N
$.
Here the spaces $E_i$ are subspaces of  $E\otimes \T $ but often they appear also in
other realizations, like the spaces $\Sym^{p+1}_0\T$ and $\Sym^{p-1}_0 \T $ in the
decomposition of $\Sym^p_0 \T \otimes \T $ considered above.

Assume that $\tilde E_i$ are
$\SO(n)$-representations isomorphic to $E_i$ and that
$\pi_i : E \otimes \T  \rightarrow \tilde E_i$ are representation morphisms with $\pi_i \circ \pi_i^\ast = c_i \, \id$
for some non-zero constants $c_i$. Then we can define projections 
$p_i : E \otimes \T  \rightarrow E_i \subset E \otimes \T $ as above by
$p_i := \frac{1}{c_i} \, \pi_i^\ast \pi_i$. From the condition on $\pi_i$
we obtain $p_i^2 = p_i^\ast \circ p_i = p_i$. Now we define two sets
of operators on sections of $EM$:
$\dd _i := \pi_i \circ \nabla : \Gamma(EM) \rightarrow \Gamma({\tilde E_i M})$
and
$P_i := p_i \circ \nabla : \Gamma (EM) \rightarrow \Gamma(EM \otimes \T M)$.

We then have the general formula
$
\dd_i^\ast \, \dd_i \;=\; c_i \, P_i^\ast \, P^{\phantom *}_i
$.
Indeed we have 
$$
\dd_i^\ast \, \dd^{\phantom *}_i \;=\;  \nabla^\ast \,\pi_i^\ast \, \pi_i \nabla
\;=\; c_i \, \nabla^\ast \, p_i  \, \nabla
\;=\; c_i \, \nabla^\ast\,  p_i^\ast \, p^{\phantom *}_i\, \nabla
\;=\; c_i \, P^\ast_i \, P_i\ .
$$
The statement of the lemma now follows from \eqref{pi12} together with \eqref{pi11}--\eqref{pi2}.
\qed

\section{Basics on Killing and conformal Killing tensors}

\begin{ede}
A symmetric tensor $K \in \Gamma(\Sym^p \T M)$ is called {\em conformal Killing tensor} if
there exists some symmetric tensor $k \in \Gamma( \Sym^{p-1}\T M)$ with
$\;
\dd  K = \LL  (k)
$.
\end{ede}

\begin{elem} \label{lemma:conformalInvariant}
    The defining equation for conformal Killing tensors is conformally
    invariant. More precisely, a section $K$ of $ \Sym^{p} \T M$ is a conformal Killing tensor
    with respect to the metric $g$, if and only if it is a conformal Killing tensor with respect to every conformally related metric $  g' = e^{2 f} g$.
\end{elem}
\proof
Let $X, Y$ be any vector fields.
Then the Levi-Civita connection $\nabla'$ for  $g' $  is given by
$$
\nabla'_X Y = \nabla_XY  + \dd f (X) Y + \dd f(Y) X - g(X, Y)\, \grad_g(f)
$$
where  $\grad_g(f)$ is the gradient of $f$ with respect
to  $g$ (cf. \cite{besse}, Th. 1.159). It immediately  follows that  for any section
$K$ of $\Sym^p \T M$ we have
$$
\nabla'_X K \;=\; \nabla_X K \,+\, p \, \dd f(X) \, K \,+\, X \cdot \grad_g(f) \, \lrcorner \, K 
\,-\,  \grad_g(f) \cdot X\,\lrcorner\, K \ .
$$
Hence we obtain for the differential 
$\dd' K =  \sum _i e'_i \cdot \nabla'_{e'_i} K = e^{-2f} \sum_i e_i \cdot \nabla'_{e_i} K$
the equation
\bea
\dd' K &=& e^{-2f} (\dd K + p \, \grad_g(f) \cdot K + \LL\,(\grad_g(f) \,\lrcorner\, K) - p \, \grad_g(f) \cdot K)\\
&=& 
e^{-2f} \dd K + \LL' \,(\grad_g(f) \,\lrcorner\, K) \ .
\eea
Hence if $K$ is conformal Killing tensor with respect to the metric $g$, i.e. $\dd K = \LL k$ for some
section $k$ of $\Sym^{p-2} \T M$, then $K$ is 
a conformal  Killing tensor with respect to the metric $g'$, too. Indeed  $\dd' K = \LL'(k + \grad_g(f) \,\lrcorner\, K)$.
\qed

\medskip
Note that $K$ is conformal Killing if and only if its trace-free part is conformal Killing. Indeed,  since $\dd $ and $\LL $ commute,
if $K = \sum_{i\ge 0} \, \LL ^i \, K_i$, with $K_i\in\Gamma(\Sym^{p-2i}_0\T M)$ is the standard decomposition of $K$, then $\dd K = \sum_{i\ge 0} \, \LL ^i \, \dd K_i$, so $\dd K$ is in the image of $\LL $ if and only if $\dd K_0$ is in the image of $\LL $.
It is thus reasonable to consider only trace-free
conformal Killing tensors.

\medskip

\begin{elem}\label{conf}
Let $K \in \Gamma(\Sym^p \T M)$, then $K$ is a conformal Killing tensor if and only if
$$
\dd  K_0 \;=\; -\, \tfrac{1}{n+2p-2}\,\LL  \, \d K_0
$$
or, equivalently, if and only if  $\; (\dd  K_0)_0 = 0$. In particular, a trace-free
symmetric tensor $K \in \Gamma(\Sym^p_0 \T M)$ is a conformal Killing tensor
if and only if $P_1(K)=0$.
\end{elem}
\proof
We write $K = \sum_{i \ge 0} \, \LL ^i \, K_i$, where $K_i$ is a section of $\Sym^{p-2i}_0\T M$.
Because of $[\LL ,\dd ] =0 $ we have
$$
\dd K  = \sum_{i\ge 0} \, \LL ^i \, \dd  \, K_i = \dd \, K_0 + \sum_{i \ge 1} \, \LL ^i \dd  K_i
= (\dd K_0 +  \tfrac{1}{n+2p-2}\, \LL  \, \delta \, K_0 ) - \tfrac{1}{n+2p-2}\, \LL  \, \delta \, K_0 + \sum_{i \ge 1} \, \LL ^i \dd  K_i
$$
We know from \eqref{dprojection} that the bracket on the right hand-side is  the trace-free part of $\dd K_0$.
Thus 
$
\dd  K_0 = -\, \tfrac{1}{n+2p-2}\,\LL  \, \d K_0
$
holds if and only if $\dd K = \LL (k)$ for some section $k$  of $\Sym^{p-1}\T M$, i.e. if and only
if $K$ is a conformal Killing tensor.
\qed

\begin{ere} \label{ft}
Since $P_1(K)$ is the projection of the covariant derivative  $\nabla K$ onto the Cartan summand
$\Sym^{p+1}_0 \T M \subset \T M \otimes \Sym^p_0 \T M$, it follows that the
defining differential equation for trace-free conformal Killing tensors is of finite
type, also called strongly elliptic (cf. \cite{kalina}). In particular the space of trace-free conformal Killing tensors is finite 
dimensional on any connected manifold. Moreover  one can show that a conformal Killing  $p$-tensor has
to vanish identically on $M$ if its first $2p$ covariant derivatives vanish
in some point of $M$, cf.  \cite{eastwood}.
\end{ere}

\begin{ede}
A symmetric tensor $K \in \Gamma( \Sym^p \T M)$ is called {\em Killing tensor} if
$\;
\dd  K = 0 \ .
$ A trace-free Killing tensor is called {\em St\"ackel tensor}.
\end{ede}

\begin{elem}
A symmetric tensor $K\in \Gamma(\Sym^p \T M)$ is a Killing tensor if and only if
the complete symmetrization of $\nabla K$ vanishes. A Killing tensor is in particular a conformal Killing tensor.
\end{elem}
\proof
Recall that $\dd  : \Gamma(\Sym^p \T M) \rightarrow \Gamma(\Sym^{p+1} \T M)$ was defined as
$
\dd  K \;=\; \sum_i  e_i \cdot \nabla_{e_i}K 
$, where $\{e_i\}$ is some local orthonormal frame. Thus
\bea
\dd K(X_1, \ldots, X_{p+1}) &=&\sum_i \sum_{\sigma \in S_p}g(e_i, X_{\sigma(1)}) \, (\nabla_{e_i} K)(X_{\sigma(2) }, \ldots , X_{\sigma(p+1)})\\
&=&
\sum_{\sigma \in S_p} \, (\nabla_{X_{\sigma(1)}} K) (X_{\sigma(2) }, \ldots , X_{\sigma(p+1)}) \ .
\eea
\qed

\medskip

\begin{epr}\label{system}
Let $K = \sum_{i \ge 0} \, \LL ^i K_i\in \Gamma(\Sym^p \T M)$, with $K_i \in \Gamma(\Sym^{p-2i}_0\T M)$ be a symmetric tensor, $p = 2l + \epsilon$ with $\epsilon = 0$ or $1$. Then $K$ is a Killing tensor if and only if the following system of equations holds:
\bea
\dd  K_0 &=& a_0 \LL  \d K_0   \ ,\\[1ex]
\dd K_1 &=& a_1 \LL  \d K_1 - a_0 \d K_0   \ ,   \\[1ex]
\vdots &=& \qquad \vdots\\[1ex]
\dd  K_l &=& a_l \LL  \d K_l - a_{l-1}\d K_{l-1} \ ,  \\[1ex]
0 &=& \d K_l \ ,
\eea
where $a_i$ are the constants of Lemma \ref{constants}. In particular, the trace-free part $K_0$ is a conformal Killing tensor.
\end{epr}
\proof
We write  $K = \sum_{i \ge 0} \, \LL ^i K_i$ with $K_i \in \Gamma(\Sym^{p-2i}_0\T M)$, then 
$\dd K = 0$ if and only if
$$
0 = \sum_{i \ge 0} \, \LL ^i \, \dd K_i = \sum_{i \ge 0} \, \LL ^i \, (\dd  K_i \,-\, a_i \LL \delta K_i) +  \, a_i \LL ^{i+1}\delta K_i
= \sum_{i \ge 0}\, \LL ^i (\dd K_i \,-\, a_i \LL \delta K_i \,+\, a_{i-1}\delta K_{i-1}) \ ,
$$
where we set  $a_{-1}=0 $  and $K_{-1}=0$ by convention.
From Lemma \ref{constants} and $[\Lambda , \delta] = 0$ it follows that $\dd K_i \,-\, a_i \LL \delta K_i \,+\, a_{i-1}\delta K_{i-1}$   is trace-free.
We conclude that $\dd K=0$ if and only if 
$
\dd K_i \,-\, a_i \LL \delta K_i \,+\, a_{i-1}\delta K_{i-1} = 0
$
for all $i$.
\qed

\medskip

\begin{exe}
For $p=2$ and $K \in \Gamma(\Sym^2\T M)$ we have $K = K_0 + 2 f  g$, for some function $f = K_1 \in C^\infty(M)$.
Then $K$ is a Killing tensor if and only if
\begin{equation}\label{2tensor}
\dd  K_0 \;=\; -\,\tfrac{1}{n+2} \, \LL  \,\d\, K_0 \qquad \mbox{and} \qquad \dd  f \;=\; \tfrac{1}{n+2}\,  \d\, K_0\ .
\end{equation}
The second equation can equivalently be written as $ \dd  \, \tr K = 2 \d K$. 
\end{exe}

\begin{exe}
For $p=3$ and $K \in \Gamma(\Sym^3\T M)$ we have $K = K_0 + \LL  \xi$, for some vector field $\xi = K_1$.
Then $K$ is a Killing tensor if and only if
$$
\dd  K_0 = - \tfrac{1}{n+4}\,  \LL  \d K_0\ , \quad  \dd  \xi = \tfrac{1}{n+4}\,  \d K_0
\quad \mbox{and} \quad \d \xi = 0 \ .
$$
\end{exe}

\medskip

\begin{ecor}\label{tracefree}
If $K \in \Gamma(\Sym^p \T M)$ is a trace-free Killing tensor, then $\d K = 0$. In other words,
St\"ackel tensors satisfy the equations $\dd  K = 0 = \d K$, or equivalently
the equations $P_1(K) = 0 = P_2(K)$.
\end{ecor}

Conversely we may ask what is possible to say about the components of divergence-free Killing tensors.
The result here is

\begin{epr}
Let $K \in \Gamma(\Sym^p \T M)$ be a divergence-free Killing tensor. Then all
components $K_i \in \Gamma(\Sym^{p-2i}_0\T M)$ of its standard decomposition
$K = \sum_{i \ge 0} \LL ^i \, K_i$  are St\"ackel tensors.
\end{epr}
\proof
We first remark that iteration of the commutator formula \eqref{commu3} gives
$[\d, \LL ^i] = - 2 i \LL ^{i-1} \dd $ for $i \ge 1$. Then $\d K = 0$ implies
$
0 = \d K_0 + \sum_{i\ge 1} \d \LL ^i K_i = \d K_0 - 2 \dd  K_1 + \LL  k_1
$
for some symmetric tensor $k_1$. Substituting $\dd  K_1$ by the second equation of 
Proposition \ref{system} we obtain
$
0 = (1 + 2 a_0) \d K_0 + \LL   k_2
$
for some symmetric tensor $k_2$. Since $\d K_0$ is trace-free and since the
coefficient $(1+2a_0) $ is different from zero for $n>2$ it follows that $\d K_0 = 0$.
But then the first equation of Proposition \ref{system} gives $\dd  K_0 = 0$. Thus
$K_0$ is a St\"ackel tensor.

We write $K = K_0 + \LL  \tilde K$ with $\tilde K = \sum_{i\ge 1} \LL ^{i-1} K_i$. Since $\dd $
commutes with $\LL $ and multiplication with $\LL $ is injective the equations
$\dd  K = 0 = \dd K_0$ imply $\dd \tilde K = 0$. Similarly we obtain 
$
0 = \d (\LL  \, \tilde K) = \LL  \d \tilde K - 2 \dd  \tilde K = \LL  \d \tilde K
$.
Thus $\d \tilde K =0$ and we can iterate the argument above.
\qed

\begin{ere}
The above proof shows that the map $\LL $ preserves the space of divergence-free
Killing tensors. 
\end{ere}

In \cite{schoebel1} the class of {\it special conformal Killing tensors} was introduced. These tensors were defined as symmetric $2$-tensors
$K \in \Gamma(\Sym^2 \T M) $ satisfying the equation
\beq\label{special}
(\nabla_X K ) (Y, Z) \;=\; k(Y)\, g(X, Z) \;+\; k(Z)\, g(X, Y)
\eeq
for some $1$-form $k$. It follows that $k= \frac12 \, \dd  \, \tr(K)$ and that
 $\dd  K = \LL  (k)$. Thus solutions of Equation \eqref{special} are automatically 
conformal Killing tensors. The tensor $k$ also satisfies  $\d K = -(n+1) \, k$ and it is easily  proved that
$\hat K := K -\tr (K)\, g$ is a Killing tensor, which is called {\it special Killing tensor} in \cite{schoebel1}.
Moreover the map $K\mapsto \hat K$ is shown to be injective and equivariant with respect to the action of the
isometry group.

We will now generalize these definitions and statements to Killing tensors of arbitrary degree.
We start with

\begin{ede}
A symmetric tensor $K\in \Gamma(\Sym^p\T M)$ is called {\em special conformal Killing tensor}
if the equation
$
\nabla_X K = X \cdot k
$
holds for all vector fields  $X$ and some symmetric  tensor $k \in \Gamma(\Sym^{p-1}\T M)$.
\end{ede}

For $p=2$, this is equivalent to Equation
\eqref{special}. Immediately from the definition it follows that $k = -\frac{1}{n+1}\d K$ and that $\dd K = \LL  k$.
Hence $K$ is in particular a conformal Killing tensor. Using the standard decomposition 
$
K = \sum_{j\ge 0} \LL ^j \, K_j
$
and
$
k = \sum_{j\ge 0} \LL ^j \, k_j
$
we can reformulate the defining equation for special conformal Killing tensors into a system
of equations for the components $K_j$ and $k_j$. Let $K$ and $k$ be symmetric tensors
as above then 
$
\nabla_X K = \sum_{j\ge 0} \LL ^j \, \nabla_X K_j
$
and by \eqref{projection} we have:
$$
X \cdot k \;=\; \sum_{j\ge 0} \LL ^j \,  (X \cdot k_j)  
\;=\;
 \sum_{j\ge 0}\left( \LL ^j \, (X \cdot k_j)_0 \;+\; \tfrac{1}{n+2(p-2-2j)} \LL ^{j+1}(\, X \lrcorner \, k_j)\right) \ .
$$
Comparing coefficients of powers of $\LL $, we conclude that $K$ is a special
conformal Killing tensor if and only if the following system of  equations is satisfied
\begin{equation}\label{sck}
\nabla_X K_j \;=\; (X\cdot k_j)_0 \;+\; \tfrac{1}{n+2(p-2j)} \, X \lrcorner \, k_{j-1},
\qquad j \ge 0 \ .
\end{equation}
With the convention $k_{-1}=0$ this contains the equation 
$
\nabla_X K_0 = (X \cdot k_0)_0
$
for $j=0$.

\medskip

\begin{ede}
A symmetric tensor $\hat K\in \Gamma(\Sym^p\T M)$ is called {\em special Killing tensor}
if it is a Killing tensor satisfying the additional equation
$
\nabla_X \hat K = X \cdot \hat k + X \lrcorner \, \hat l
$
for all vector fields $X$ and some symmetric tensors $\hat k \in \Gamma(\Sym^{p-1} \T M)$
and $\hat l \in \Gamma(\Sym^{p+1} \T M)$.
\end{ede}

From the definition it follows directly that the tensors $\hat k$ and $\hat l$ are related by the equations
$$
\hat l \;=\; -\, \tfrac{1}{p+1} \LL  \hat k
\qquad
\mbox{and}
\qquad
\d \hat K = -(n+p-1) \hat k - \Lambda \hat l \ .
$$
Hence $\hat K$ is a special Killing tensor if and only if for all vector fields $X$ we have
$$
\nabla_X \hat K \;=\; X \cdot \hat k - \tfrac{1}{p+1} X \lrcorner \,\LL  \hat k 
\;=\; \tfrac{p-1}{p+1} X \cdot \hat k - \tfrac{1}{p+1} \LL  X \lrcorner \, \hat k .
$$

As for special conformal Killing tensors we can reformulate the defining equation for special Killing tensors
as a system of equations for the components $\hat K_j$ and $\hat k_j$. We find

\bea
\sum_{j \ge 0} \LL ^j \nabla_X \hat K_j &=&
\sum_{j \ge 0}  \tfrac{p-1}{p+1} \, X \cdot \LL ^j \, \hat k_j \;-\; \tfrac{1}{p+1} \,
\LL \, X \lrcorner \, \LL ^j \, \hat k_j \\[1ex]
&=&
\sum_{j \ge 0}  \tfrac{p-1 -2j }{p+1} \,  \LL ^j \, X \cdot  \hat k_j \;-\;\tfrac{1}{p+1} \,
\LL ^{j+1} X \lrcorner  \, \hat k_j \\[1ex]
&=&
\sum_{j \ge 0}  \tfrac{p-1 -2j }{p+1} \,  \LL ^j \, (X \cdot  \hat k_j)_0 \;+\;
\left(
\tfrac{p-1-2j}{(p+1)(n+2(p-2j-2))} \,-\, \tfrac{1}{p+1}
\right) \LL ^{j+1} X \lrcorner  \, \hat k_j\\[1ex]
&=&
\sum_{j \ge 0}  \tfrac{p-1 -2j }{p+1} \,  \LL ^j \, (X \cdot  \hat k_j)_0 \;-\;
\tfrac{n+p-2j-3   }{(p+1)(n+2(p-2j-2))} 
\,  \LL ^{j+1} X \lrcorner  \, \hat k_j
\eea
Hence $\hat K$ is a special Killing tensor if and only if the components $\hat K_j$ and $\hat k_j$ of
the standard decomposition satisfy the equations
\begin{equation}\label{sk}
\nabla_X \hat K_j \;=\;  \tfrac{p-1 -2j }{p+1} \,(X \cdot  \hat k_j)_0
\;-\; \tfrac{n+p-2j-1}{(p+1)(n+2(p-2j))} \, X \lrcorner  \, \hat k_{j-1} 
\qquad j \ge 0\ .
\end{equation}
With the convention $\hat k_{-1} = 0$ we have for $j=0$ the equation
$
\nabla_X \hat K_0 = \frac{p-1}{p+1} (X \cdot \hat k_0)_0
$.

\medskip

For a given special conformal Killing tensor $K$ we now want to modify its components
$K_j$ by  scalar factors in order to obtain a special Killing tensor $\hat K$. This will 
generalize the correspondence between special conformal and special Killing $2$-tensors
obtained in \cite{schoebel1}.

We are looking 
for constants $a_j$, such that $\hat K = \sum_{j\ge 0} \hat K_j$ with
$\hat K_j = a_j K_j$ is a special Killing tensor, where $\hat k = \sum_{j\ge 0} \hat k_j$
with $\hat k_j = b_j k_j$ for a other set of constants $b_j$. Considering first the
special Killing equation for $j=0$ we have
$
\nabla_X a_0K_0 = \frac{p-1}{p+1} (X \cdot b_0 k_0)_0
$.
Hence we can define $a_0=1$ and $b_0 = \frac{p+1}{p-1}$. Then writing  \eqref{sk}
with the modified tensors $\hat K_j $ and $\hat k_j$ and comparing it with \eqref{sck}
multiplied by $a_j$, we get the condition 
$
a_j = \frac{p-1-2j}{p+1}b_j
$
for the first summand on the right hand side and 
$$
\tfrac{a_j}{n+2(p-2j)} \;=\; -\, \tfrac{(n+p-2j-1) }{(p+1)(n+2(p-2j))}\,b_{j-1}
\;=\;
-\,\tfrac{(n+p-2j-1) }{(p+1)(n+2(p-2j))} \, \tfrac{p+1}{p+1-2j} \, a_{j-1} \ ,
$$
for the second summand. Hence
$
a_j = - \, \tfrac{n+p-2j-1}{p+1-2j} \, a_{j-1}
$
and in particular
$
a_1 = -\, \tfrac{n+p-3}{p-1}
$.
By this recursion formula the coefficients $a_j$ and $b_j$ are completely determined and
indeed $\hat K$ defined as above is a special Killing tensor.

As an example we consider the case $p=2$. Let $K = K_0 + \LL  K_1$ be a special conformal
Killing $2$-tensor. Then
$
\hat K = K_0 + a_1 \LL  K_1 = K_0 -(n-1)\LL  K_1 = K - n \LL  K_1 = K - \tr(K) g
$.
Indeed $\tr(K) = \Lambda K = 2n K_1$ and $\LL  = 2g$. Hence we obtain for special
conformal Killing $2$-tensors the same correspondence as in \cite{schoebel1}.

Special conformal Killing $2$-tensors have the important property of being integrable. Indeed
their associated Nijenhuis tensor vanishes (cf. \cite{schoebel1}, Prop. 6.5). Recall that if $A$ is any endomorphism field on $M$, its {\it Nijenhuis tensor} is defined by
\bea
N_A(X, Y)
&=&
 -\, A^2 [X,Y] \;+\; A [AX, Y] \;+\; A [X, AY] \;-\; [AX, AY]\\[1ex]
&=&
\; A \, (\nabla_XA)\, Y    \;-\; A \, (\nabla_Y A) X \;-\; (\nabla_{AX} A ) Y \;+\; (\nabla_{AY} A )X \ .
\eea

\begin{elem}
Any special conformal Killing $2$-tensor has vanishing Nijenhuis tensor and thus
is integrable.
\end{elem}
\proof
Let $A$ be a special conformal Killing $2$-tensor. Then $\nabla_XA = X \cdot \xi $
for some vector field $\xi$. Then
$(\nabla_X A)Y = g(X, Y) \, \xi + g(\xi, Y) X $ and it follows
\bea
N_A(X, Y) &=&
g(X, Y) \, A \xi \;+\; g(\xi, Y)  A X
\;-\; g(Y, X) \, A \xi \;-\; g(\xi, X) A Y \\[1ex]
&&
\qquad - \; g(AX, Y)\, \xi \;-\; g(\xi, Y) \, AX \;+\; g(AY, X) \, \xi \;+\; g(\xi, X)\, A Y \;=\; 0 \ .
\eea
\qed

It is easy to check that special Killing $2$-tensors associated to special conformal Killing tensors as above, have non-vanishing Nijenhuis tensor, therefore the statement in \cite{schoebel1}, Prop. 6.5 is not valid for special Killing tensors.

\medskip

From Proposition \ref{system} it follows that St\"ackel tensors are characterized among trace-free tensors by the vanishing of the two operators
$P_1$ and $P_2$. It is natural to consider trace-free symmetric tensors with other
vanishing conditions. Here we mention two other cases:

\begin{elem}
For a trace-free symmetric tensor  $K \in \Gamma( \Sym^p_0 \T M)$ 
the relations $P_1(K)= 0$ and $P_3(K)=0$ hold if and only if there exists a section
$k$ of $\Sym^{p-1} \T M$ with $\nabla_XK = (X\cdot k)_0$. In this
case $k$ is uniquely determined and
$\;
\nabla_XK =  -\, \tfrac{n+2p-4}{(n+2p-2)(n+p-3)} \, (X \cdot \d K)_0
$.
\end{elem}

It follows from \eqref{sck} for $j=0$ that the trace-free part $K_0$ of a  special conformal Killing tensor
$K$ satisfies $P_1(K_0)= 0$ and $P_3(K_0)=0$.

For $p=2$ consider a section $K_0$ of
$\Sym^2_0 \T M$ with $P_1(K_0)= 0$ and $P_3(K_0)=0$. Then there exists a $1$-form $k$ with
$
\nabla_X K_0 = (X\cdot k)_0 
= 
X \cdot k - \tfrac{1}{n+2p-4} \, \LL  (X \lrcorner \, k)
=
X \cdot k  - \tfrac{2}{n} \,  k(X)\, g$.
Substituting two tangent vectors $Y$ and $Z$ we obtain the equation
\beq\label{special1}
(\nabla_X K_0) (Y, Z) = g(X,Y)\, k(Z) \;+\; g(X, Z) \, k(Y) \;-\; \tfrac2n\, k(X)\,g(Y, Z) \ .
\eeq

For the sake of completeness, we also consider the case of trace-free symmetric tensors
$K \in \Gamma(\Sym^p_0 \T M) $ satisfying the vanishing conditions $P_2(K) = 0 = P_3(K)$. Equivalently the
covariant derivative $\nabla K$ is completely symmetric, i.e. 
$
\nabla K \in \Gamma( \Sym^{p+1}_0 \T M)
$.
This can also be written as
$
(\nabla_X K) (Y, \ldots ) = (\nabla_Y K) (X, \ldots )
$
for all tangent vectors $X$, $Y$
or as $\dd^\nabla K = 0$, where $K$ is considered as a $1$-form with values
in $\Sym^{p-1}_0 \T M$, i.e. a section of $\T ^*M \otimes \Sym^{p-1}_0 \T M$.
For $p=2$, a tensor $K$ with $\dd^\nabla K = 0$ is called a
{\it Codazzi tensor}.


\section{Examples of manifolds with Killing tensors}

Any parallel tensor is tautologically a Killing tensor. By the conformal invariance of the conformal Killing equation, 
a parallel tensor defines conformal Killing tensors for any conformally related metric. These are in general
no Killing tensors. There are several explicit  constructions of symmetric Killing tensors which we will 
describe in the following subsections.


\subsection{Killing tensors on the sphere}

Let $\mathrm{Curv}(n+1)$ denote the space of algebraic curvature tensors on $\RM^{n+1}$.  Then
any symmetric Killing $2$-tensor $K$ on $\mathbb{S}^n$ is given in a point $p\in \mathbb{S}^n$ by
$
K(X, Y) = R(X, p, p, Y)
$
for some algebraic curvature tensor $R \in \mathrm{Curv}(n+1)$. The subspace of Weyl curvature
tensors in  $\mathrm{Curv}(n+1)$ corresponds to the trace-free (and  hence divergence-free)
Killing tensors on $\mathbb{S}^n$, cf. \cite{mclena}.

The dimension of the space of Killing tensors on $\mathbb{S}^n$ gives an upper bound for this
dimension on an arbitrary Riemannian manifold \cite{thompson},  Theorem 4.7.


\subsection{Symmetric products of Killing tensors}\label{4.2}

Let $(M^n, g)$ be a Riemannian manifold with two Killing vector fields $\xi_1, \xi_2$. We define a symmetric
$2$-tensor $h$ as the symmetric product $h := \xi_1 \cdot \xi_2$. Then $h$ is a Killing tensor. Indeed
$
\dd  (\xi_1 \cdot \xi_2) = (\dd \xi_1) \cdot \xi_2 + \xi_1 \cdot (\dd\xi_2) = 0
 $
holds because of Lemma \ref{derivation}. 
More generally the symmetric product of  Killing tensors defines again a Killing tensor. Conversely, it is known
that on manifolds of constant sectional curvature, any Killing tensor can be written as a linear
combination of symmetric products of Killing vector fields, cf.  \cite{thompson},  Theorem 4.7.

\medskip

\begin{elem}\label{killing}
For any two Killing vector fields $\xi_1, \xi_2$ it holds that
$$
\d (\xi_1 \cdot \xi_2) \;=\; \dd  \, g( \xi_1, \xi_2)   \ .
$$
\end{elem}
\proof
We compute $\d (\xi_1 \cdot \xi_2) = - \sum e_i \lrcorner \nabla_{e_i} (\xi_1 \cdot \xi_2)
= -\sum e_i \lrcorner ((\nabla_{e_i} \xi_1) \cdot \xi_2   + \xi_1 \cdot (\nabla_{e_i} \xi_2))$.
Since the Killing vector fields $\xi_1, \xi_2$ are divergence free we obtain
$
\d (\xi_1 \cdot \xi_2) = - \nabla_{\xi_1} \xi_2 - \nabla_{\xi_2} \xi_1
$.
Using again that $\xi_1, \xi_2$ are Killing vector field we have
$$
X(g(\xi_1, \xi_2)) \;=\; g(\nabla_X\xi_1, \xi_2) \;+ \;g(\xi_1, \nabla_X\xi_2) \;=\; -\, g(\nabla_{\xi_2} \xi_1 \,+\, \nabla_{\xi_1} \xi_2, X )
\;=\; g ( \d (\xi_1 \cdot \xi_2) , X) \ .
$$
\qed

\begin{ecor}
Let $\xi_1, \xi_2$ be two Killing vector fields with constant scalar product $g(\xi_1, \xi_2 )$. Then
$\xi_1 \cdot \xi_2$ is a divergence free Killing tensor and
$$
h \;:=\; \xi_1 \cdot \xi_2 \;-\; g( \xi_1, \xi_2 ) \, \tfrac{2}{n}\, g 
 $$
 is a trace-free, divergence-free Killing $2$-tensor.
\end{ecor}
\proof
Indeed the trace of the symmetric endomorphism
$h = \xi_1 \cdot \xi_2 \in \Gamma(\Sym^2 \T M)$ is given as $\tr (h) = 2 \,g(\xi_1, \xi_2 )$. 
Hence the tensor $h$ defined as above is a trace-free and divergence-free Killing tensor.
\qed

\begin{exe}
Let $(M^n, g, \xi)$ be a Sasakian manifold. Then $\xi \cdot \xi - \frac{2}{n} g$ is a trace-free Killing $2$-tensor. On a $3$-Sasakian manifold one has three pairwise orthogonal
Killing vector fields of unit length defining  a six-dimensional space of trace-free Killing $2$-tensors.
\end{exe}

\begin{exe}
On spheres $\SM^n$ with $n\ge 3$ one has pairs of orthogonal Killing vector fields,
defining trace-free Killing $2$-tensors. Indeed, every pair of anti-commuting skew-symmetric matrices of dimension $n+1$ defines a pair of orthogonal Killing vector fields on $\SM^n$.
\end{exe}


\subsection{Killing tensors from Killing forms}\label{4.3}

There is a well-known relation between Killing forms and Killing tensors (e.g. cf. \cite{carter}).
Let $u \in \Omega^p(M)$ be  a Killing form, i.e. a $p$-form satisfying the equation $X \,\lrcorner\, \nabla_X u =0$ for any
tangent vector $X$. We define a symmetric bilinear form $K^u$ by $K^u(X, Y) = g(X \lr u, Y \lr u) $. Then $K^u$ is a symmetric
Killing $2$-tensor (for $p=2$ this fact was also remarked in \cite{andrei15}, Rem. 2.1). Indeed it suffices to show $(\nabla_X K^u)(X,X)=0$ for any tangent vector $X$, which is immediate:
$$
(\nabla_XK^u)(X,X) \;=\; \nabla_X(K^u(X,X)) \;-\; 2\, K^u(\nabla_XX, X) \;=\; 2\, g(X \lr \nabla_Xu, X \lr u) \;=\; 0 \ .
$$
Since Killing $1$-forms are dual to Killing vector fields, this construction generalizes the the one described in Section \ref{4.2}.
If $u$ is a Killing $2$-form considered as skew-symmetric endomorphism, then the associated symmetric 
Killing tensor $K^u$ is just $-u^2$. In this case $K^u$ commutes with the 
Ricci tensor, since the same is true for the skew-symmetric endomorphisms corresponding to Killing $2$-forms, cf. \cite{carter}.
Examples of manifolds with Killing forms are: the standard sphere $\mathbb{S}^n$, Sasakian, 3-Sasakian, nearly 
K\"ahler
or weak $\mathrm{G}_2$ manifolds \cite{uwe}. 

A related construction appears in the work of 
V. Apostolov, D.M.J. Calderbank and P. Gauduchon \cite{acg13}, in particular Appendix B.4. They prove a 1--1 relation between
symmetric $J$- invariant Killing $2$-tensors on K\"ahler surfaces and Hamiltonian $2$-forms. In contrast to
Killing forms there are many examples known of Hamiltonian $2$-forms, thus providing a rich source of 
symmetric Killing tensors.


\subsection{The Ricci curvature as Killing tensor}

Special examples of Killing tensors arise as Ricci curvature of a Riemannian metric. Notice that if $\Ric$ is a Killing tensor then the scalar curvature $\scal$ is constant. Riemannian manifolds whose Ricci tensor is Killing were studied in 
\cite{besse} Ch. 16.G, as a class of generalized Einstein manifolds. In the same context this condition was originally 
discussed by A. Gray in \cite{gray78}. It can be  shown shown that all D'Atri spaces, i.e. Riemannian manifolds whose local geodesic symmetries preserve, up to a sign, the volume element,  have Killing Ricci tensor
 (cf.  \cite{besse}, \cite{datri}). This provides a wide class of examples, many of
 them with non-parallel Ricci tensor. In particular naturally reductive spaces are D'Atri, and thus have 
 Killing Ricci tensor. Here we want to present a direct argument.

\begin{epr}
The Ricci curvature of a naturally reductive space is a Killing tensor.
\end{epr}
\proof
Naturally reductive spaces are characterized by the existence of a metric connection $\bar \nabla$ with  
skew-symmetric, $\bar\nabla$-parallel torsion $T$ and parallel curvature $\bar R$. This gives in particular
the following equations
$$
g(T_XY, Z) + g(Y,T_XZ) = 0
\qquad\mbox{and}\qquad
T_XY + T_YX =0 \ .
$$
The condition $\bar\nabla \bar R = 0$ can be rewritten as the following equation for the Riemannian curvature $R$
$$
(\nabla_XR)_{Y,Z} \;=\; [T_X, R_{Y,Z}]  \;-\; R_{T_XY, Z} \;-\; R_{Y, T_XZ} \ .
$$
The  Ricci curvature is defined as $\Ric(X,Y) = \sum g(R_{X, e_i}e_i, Y )$, where
$\{e_i\}$ is an orthonormal frame. Its covariant derivative is given by
$$
(\nabla_Z \Ric)(X, Y)
\;=\; \sum g((\nabla_ZR)_{X, e_i} e_i, Y) \ .
$$
The  Ricci curvature $\Ric \in \Gamma(\Sym^2 \T M)$ is a Killing tensor if and only if 
$(\nabla_X \Ric) (X,X)=0$ for all tangent vectors $X$, which is equivalent to
$$
\sum g((\nabla_X R)_{X, e_i} e_i, X) \;=\; 0 \ .
$$ 
If $R$ is the Riemannian curvature tensor of naturally reductive metric this curvature
expression can be rewritten as
\bea
\sum g((\nabla_X R)_{X, e_i} e_i, X) 
&=&
\sum 
g([\,T_X, R_{X, e_i} ] \,e_i \,-\, R_{T_XX, e_i} e_i \,-\, R_{X, T_Xe_i}e_i, \, X) \\[1ex]
&=&
\sum
g(T_X \, R_{X, e_i} e_i, \, X) \;-\; g(R_{X, e_i} \, T_Xe_i, \, X)
\;-\;g ( R_{X, T_Xe_i} e_i, \, X)\\[1ex]
&=&
\sum
- g(R_{X, e_i} e_i, T_XX) \;-\; 2\,  g(R_{X, e_i} \, T_Xe_i, \, X)\\[1ex]
&=&
0
\eea
Here we used  $T_XX=0$ and the equation $g(R_{X, e_i} \, T_Xe_i, \, X) = 0$, which holds because of 
\bea
g(R_{X, e_i} \, T_Xe_i, \, X) &=& -\, g(R_{e_i, T_X e_i} \, X, \, X) \;-\; g(R_{T_X e_i , X} e_i , X)
\;=\; g(R_{X, T_Xe_i}e_i, \, X) \\[1ex]
&=&
-\,g(R_{X, e_i} \, T_Xe_i, \, X) \ .
\eea
\qed

\bigskip

\noindent
We define a modified Ricci tensor $\widetilde \Ric \in \Sym^2 \T M$ as 
$\;
\widetilde  \Ric := \Ric - \tfrac{2 \, \scal}{n+2}\,\Id \ .
$
If $ \Ric$ is a Killing tensor then the same is true for $ \widetilde \Ric$, but not conversely.

\begin{elem}
The modified Ricci tensor $\widetilde \Ric$ is Killing tensor if and only if it is a
conformal Killing tensor.
Moreover $\widetilde \Ric $ is Killing if and only if
$\;
(\nabla_X \Ric)(X, X) = \tfrac{2}{n+2}\, X(\scal)\, g(X,X) \;
$
for   all vector  fields $X$.
\end{elem}
\proof
A symmetric $2$-tensor is a Killing tensor if and only if the two equations of \eqref{2tensor} are
satisfied. The first equation characterizes conformal Killing $2$-tensors. Hence we only have
to show that the second equations holds for $\widetilde \Ric$, i.e. we have to show that
$\dd \, \tr (\widetilde \Ric) = 2 \,\d\, \widetilde \Ric$. Since $\, \tr (\widetilde \Ric) = \frac{2-n}{\;n+2} \, \scal \,$,
the well-known relation 
$\d \, \Ric  = - \frac12 \, \dd  \, \scal $ implies
$$
\d \, \widetilde \Ric = -\tfrac12 \, \dd  \, \scal + \tfrac{2}{n+2}\,\dd  \, \scal = \tfrac{-n + 2}{\;2(n+2)}\, \dd  \, \scal
,\qquad
\dd \, \tr (\widetilde \Ric ) = \tfrac{2-n}{n+2} \, \dd  \, \scal \ .
$$
Hence $\dd  \, \tr (\widetilde \Ric) = 2 \,\d\, \widetilde \Ric$ and the modified Ricci tensor $\widetilde \Ric$
is a Killing tensor if it is a conformal Killing tensor. The other direction and the equation for $\Ric$ are obvious.
\qed

\begin{ere}
A. Gray introduced  in \cite{gray78} the notation $\mathcal A$, for  the class of Riemannian manifolds with Killing Ricci tensor,
and $\mathcal C$ for the class of Riemannian manifolds  with constant scalar curvature.
The class of Riemannian manifolds whose modified Ricci tensor is Killing was studied by W. Jelonek in \cite{jelonek99} under the name $\mathcal A \oplus \mathcal C^\perp$. Finally we note that there are manifolds with Killing
Ricci tensors which are neither homogeneous nor D'Atri. Examples were constructed by H. Pedersen and  P. Tod
in \cite{tod} and by W. Jelonek in \cite{jelonek98}.
\end{ere}


\section{Conformal Killing tensors on Riemannian products}

Let $(M_1,g_1)$ and $(M_2,g_2)$ be two compact Riemannian manifolds. The aim of this section is to prove the following result, which reduces the study of conformal Killing 2-tensors on Riemannian products $M:=M_1\times M_2$ to that of Killing tensors on the factors. 

\begin{ath}\label{product}
Let $h\in \Gamma( \Sym^2_0(\T M))$ be a trace-free conformal Killing tensor. Then there exist Killing tensors $K_i\in  \Gamma(\Sym^2(\T M_i)),
\, i=1,2$, and Killing vector fields $\xi_1,\ldots ,\xi_k$ on $M_1$ and $\zeta_1,\ldots,\zeta_k$ on $M_2$ such that 
$$h=(K_1+K_2)_0+\sum_{i=1}^k\xi_i\cdot\zeta_i \ .$$
Conversely, every such tensor on $M$ is a trace-free  conformal Killing tensor.
\end{ath}

\begin{proof} We denote by $n$, $n_1$ and $n_2$ the dimensions of $M$, $M_1$ and $M_2$.
Consider the natural decomposition $h=h_1+h_2+\phi$, where $h_1$, $h_2$ and $\phi$ are sections of $\pi_1^*\Sym^2(\T M_1)$, $\pi_2^*\Sym^2(\T M_2)$, and $\pi_1^*\T M_1\otimes\pi_2^*\T M_2$ respectively. We consider the lifts to $M$ of the operators $\dd_i,\ \delta_i,\ \LL_i,\ \Lambda_i$ on the factors. Clearly two such operators commute if they have different subscripts, and satisfy the relations  \eqref{commu} and \eqref{commu3} if they have the same subscript.
We define $f_i:=\Lambda_i(h_i)$. Since $h$ is trace-free, we have $f_1+f_2=0$. The conformal Killing equation 
$$\dd h=-\frac1{n+2}\LL \delta h$$ 
reads
\begin{eqnarray*}\dd_1h_1+\dd _1h_2+\dd _1\phi+\dd _2h_1+\dd _2h_2+\dd _2\phi&=&-\frac1{n+2}\LL _1(\delta_1 h_1+\delta_1\phi+\delta_2 h_2+\delta_2\phi)\\
&&-\frac1{n+2}\LL _2(\delta_1 h_1+\delta_1\phi+\delta_2 h_2+\delta_2\phi)\ .
\end{eqnarray*}
Projecting this equation onto the different summands of $\Sym^3(\T M)$ yields the following system:
\begin{equation}\label{sys}
\begin{cases}
\dd _1h_1&=-\frac1{n+2}\LL _1(\delta_1 h_1+\delta_2\phi)\\
\dd _2h_2&=-\frac1{n+2}\LL _2(\delta_2 h_2+\delta_1\phi)\\
\dd _1h_2+\dd _2\phi&=-\frac1{n+2}\LL _2(\delta_1 h_1+\delta_2\phi)\\
\dd _2h_1+\dd _1\phi&=-\frac1{n+2}\LL _1(\delta_2 h_2+\delta_1\phi)
\end{cases}
\end{equation}
Applying $\Lambda_1$ to the first equation of \eqref{sys} and using \eqref{commu3} gives
$$-2\delta_1h_1+\dd _1f_1=-\frac{2(n_1+2)}{n+2}(\delta_1h_1+\delta_2\phi)\ ,$$
whence 
\begin{equation}\label{d2phi2} \delta_1h_1=\frac{n_1+2}{n_2}\delta_2\phi+\frac{n+2}{2n_2}\dd _1f_1\ .
\end{equation}
Similarly, applying $\Lambda_2$ to the second equation of \eqref{sys} gives
\begin{equation}\label{d1phi2} \delta_2h_2=\frac{n_2+2}{n_1}\delta_1\phi+\frac{n+2}{2n_1}\dd _2f_2\ .
\end{equation}
Replacing $\delta_i h_i$ in the right hand side of \eqref{sys} using \eqref{d2phi2} and \eqref{d1phi2}, yields
\begin{equation}\label{sys2}
\begin{cases}
\dd _1h_1&=-\frac1{2n_2}\LL _1(2\delta_2 \phi+\dd _1f_1)\\
\dd _2h_2&=-\frac1{2n_1}\LL _2(2\delta_1 \phi+\dd _2f_2)\\
\dd _1h_2+\dd _2\phi&=-\frac1{2n_2}\LL _2(2\delta_2 \phi+\dd _1f_1)\\
\dd _2h_1+\dd _1\phi&=-\frac1{2n_1}\LL _1(2\delta_1 \phi+\dd _2f_2)
\end{cases}
\end{equation}
We now apply $\delta_2$ to the third equation of \eqref{sys2} and use \eqref{commu3} and \eqref{d1phi2} together with the fact that $f_2=-f_1$ to compute:
\bea \delta_2\dd_2\phi&=&-\delta_2\dd _1h_2+\frac1{n_2}\dd _2(2\delta_2 \phi+\dd _1f_1)\\
&=&-\dd _1\left(\frac{n_2+2}{n_1}\delta_1\phi+\frac{n+2}{2n_1}\dd _2f_2\right)+\frac2{n_2}\dd _2\delta_2 \phi+\frac1{n_2}\dd _2\dd _1f_1\\
&=&-\frac{n_2+2}{n_1}\dd _1\delta_1\phi+\frac2{n_2}\dd _2\delta_2 \phi+\frac{n(n_2+2)}{2n_1n_2}\dd _2\dd _1f_1\ .
\eea
Similarly, applying $\delta_1$ to the fourth equation of \eqref{sys2} and using \eqref{commu3} and \eqref{d2phi2} yields:
\begin{equation}\label{dd}\delta_1\dd_1\phi=-\frac{n_1+2}{n_2}\dd_2\delta_2\phi+\frac2{n_1}\dd_1\delta_1 \phi+\frac{n(n_1+2)}{2n_1n_2}\dd _2\dd _1f_2\ .\end{equation}
In order to eliminate the terms involving $f_1$ and $f_2$ in these last two formulas, we multiply the first one with $n_1+2$ and add it to the second one multiplied with $n_2+2$, which yields
$$(n_1+2)\delta_2\dd_2\phi+(n_2+2)\delta_1\dd_1\phi+(n_2+2)\dd_1\delta_1 \phi+(n_1+2)\dd_2\delta_2 \phi=0\ ,$$
which after a scalar product with $\phi$ and integration over $M$ gives $\dd_1\phi=\dd_2\phi=0$ and $\delta_1\phi=\delta_2\phi=0$. Plugging this back into \eqref{dd} also shows that $\dd_1\dd_2 f_2=0$. In other words, there exist functions $\f_i\in C^\infty(M_i)$ such that $f_2=\f_1+\f_2$, and correspondingly 
$f_1=-\f_1-\f_2$ (we identify here $\f_i$ with their pullbacks to $M$ in order to simplify the notations). 

The system \eqref{sys2} thus becomes:

\begin{equation}\label{sys3}
\begin{cases}
\dd _1h_1&=-\frac1{2n_2}\LL _1\dd _1f_1=\frac1{2n_2}\LL _1\dd _1\f_1\\
\dd _2h_2&=-\frac1{2n_1}\LL _2\dd _2f_2=-\frac1{2n_1}\LL _2\dd _2\f_2\\
\dd _1h_2&=-\frac1{2n_2}\LL _2\dd _1f_1=\frac1{2n_2}\LL _2\dd _1\f_1\\
\dd _2h_1&=-\frac1{2n_1}\LL _1\dd _2f_2=-\frac1{2n_1}\LL _1\dd _2\f_2
\end{cases}
\end{equation}

This system shows that the tensors $$K_1:=h_1+\frac1{2n_1}\LL _1\f_2-\frac1{2n_2}\LL _1\f_1\qquad \hbox{and}\qquad K_2:=h_2+\frac1{2n_1}\LL _2\f_2-\frac1{2n_2}\LL _2\f_1$$ verify $\dd_1 K_2=\dd_2 K_1=0$ and  $\dd_1 K_1=\dd_2 K_2=0$. The first two equations show that $K_1$ and $K_2$ are pull-backs of symmetric tensors on $M_1$ and $M_2$, and the last two equations show that these tensors are Killing tensors of the respective factors $M_1$ and $M_2$. Moreover, we have 
$$K_1+K_2=h_1+h_2+\LL\left(\frac1{2n_1}\f_2-\frac1{2n_2}\f_1\right)$$
and thus $(K_1+K_2)_0=h_1+h_2$.

Finally, we claim that $\dd _1\phi=0$ and $\dd _2\phi=0$ imply that $\phi$ has the form stated in the theorem. Let $T_1:=\pi_1^*(\T M_1)$ and $T_2:=\pi_2^*(\T M_2)$ denote the pull-backs on $M$ of the tangent bundles of the factors. Then $\phi$ is a section of $T_1\otimes T_2$, and $\dd _1\phi=0$ and $\dd _2\phi=0$ imply that $\nabla\phi=\nabla^{M_1}\phi+\nabla^{M_2}\phi$ is a section of $\Lambda^2T_1\otimes T_2+T_1\otimes\Lambda^2 T_2$. In some sense, one can view $\phi$ as a Killing vector field on $M_1$ twisted with $T_2$, and also as a Killing vector field on $M_2$ twisted with $T_1$. Let us denote by $\phi_1:=\nabla^{M_1}\phi \in\Gamma(\Lambda^2T_1\otimes T_2)$, $\phi_2:=\nabla^{M_2}\phi \in\Gamma(T_1\otimes\Lambda^2 T_2)$ and $\phi_3:=\nabla^{M_1}\nabla^{M_2}\phi=\nabla^{M_2}\nabla^{M_1}\phi\in \Gamma(\Lambda^2T_1\otimes \Lambda^2 T_2)$.
The usual Kostant formula for Killing vector fields immediately generalizes to
$$\nabla_{X_1}\phi_1=R_{X_1}\phi,\qquad \nabla_{X_2}\phi_2=R_{X_2}\phi,\qquad \nabla_{X_1}\phi_3=R_{X_1}\phi_2,\qquad \nabla_{X_2}\phi_3=R_{X_2}\phi_1,$$
where for any vector bundle $F$, $i\in\{1,2\}$, and $X\in T_i$, $R_X:T_i\otimes F\to \Lambda^2 T_i\otimes F$ is defined by
$R_X(Y\otimes\sigma):=R_{X,Y}\otimes\sigma$.

We thus get a parallel section $\Phi:=(\phi,\phi_1,\phi_2,\phi_3)$ of the vector bundle 
$$E:=(T_1\otimes T_2)\oplus (\Lambda^2T_1\otimes T_2)\oplus(T_1\otimes\Lambda^2 T_2)\oplus(\Lambda^2T_1\otimes \Lambda^2 T_2)$$
with respect to the connection defined on vectors $X_i\in T_i$ by
$$\tilde\nabla_{X_1}(\phi,\phi_1,\phi_2,\phi_3):=(\nabla_{X_1}\phi-\phi_1(X_1),\nabla_{X_1}\phi_1-R_{X_1}\phi,\nabla_{X_1}\phi_2-\phi_3(X_1),\nabla_{X_1}\phi_3-R_{X_1}\phi_2)$$
and
$$\tilde\nabla_{X_2}(\phi,\phi_1,\phi_2,\phi_3):=(\nabla_{X_2}\phi-\phi_2(X_2),\nabla_{X_2}\phi_1-\phi_3(X_2),\nabla_{X_2}\phi_2-R_{X_2}\phi,\nabla_{X_2}\phi_3-R_{X_2}\phi_1).$$
We now define for $i=1,2$ the connections $\tilde\nabla^i$ on $E_i:=\T M_i\oplus \Lambda^2\T M_i$ by 
$$\tilde\nabla^i_{X_i}(\alpha_i,\beta_i):=(\nabla^{M_i}_{X_i}\alpha_i-\beta_i(X_i),\nabla^{M_i}_{X_i}\beta_i-R^{M_i}_{X_i}(\alpha_i)),$$
and notice that $E=\pi_1^*(E_1)\otimes\pi_2^*(E_2)$ and that $\tilde\nabla$ coincides with the tensor product connection induced by $\tilde\nabla^1$ and $\tilde\nabla^2$ on $E$. It follows that the space of $\tilde\nabla$-parallel sections of $E$ is the tensor product of the spaces of $\tilde\nabla^1$-parallel sections of $E_1$ and of $\tilde\nabla^2$-parallel sections of $E_2$. Taking the first component of these sections yields the desired result.
\end{proof}

\begin{ere}
Note that the compactness assumption in Theorem \ref{product} is essential. There are many non compact
products, with  trace-free conformal Killing 2-tensors which are not defined by Killing tensors of the factors. The simplest example is the
flat space $\RM^n$. 
\end{ere}


\section{Weitzenb\"ock formulas}

Let $(M^n, g)$ be an oriented Riemannian manifold with Riemannian curvature tensor $R$.
The curvature operator $\mathcal R : \Lambda^2 \T M \rightarrow \Lambda^2 \T M$ is defined
by $g(\mathcal R (X \wedge Y), Z\wedge V) = R(X,Y,Z, V)$. 
With this convention we have
$\mathcal R = - \, \Id$ on the standard sphere.

 Let 
$P=P_{\SO(n)}$ be the frame bundle and $EM$ a vector bundle associated to $P$ via a
$\SO(n)$-representation $\rho: \SO(n) \rightarrow \Aut (E)$.
Then the curvature endomorphism $q(R) \in \End \, EM$ is  defined as
$$
q(R) \;:=\; \frac12 \, \sum_{i, j} (e_i \wedge e_j)_\ast \circ \mathcal R(e_i \wedge e_j)_\ast \ .
$$
Here $\{e_i\}, i = 1, \ldots n$, is a local orthonormal frame and for $X\wedge Y \in \Lambda^2 \T M $
we define $(X \wedge Y)_\ast = \rho_\ast (X \wedge Y)$, where $\rho_\ast: \so(n) \rightarrow \End E$ is the differential
of $\rho$. In particular, 
the standard action of $\Lambda^2\T M$ on $\T M $ is written as
$
(X\wedge Y)_\ast\, Z \;=\; g( X,\,Z)\,Y \;-\; g(Y,\, Z) \, X = (Y \cdot X \,\lrcorner - X \cdot Y \, \lrcorner \,) Z
$.
This is compatible with
$$
g( (X \wedge Y)_\ast Z, V ) \;=\; g( X \wedge Y, Z \wedge V ) \;=\; g(X,Z) \, g(Y, V) \;-\; g( X, V)\, g( Y, Z)  \ .        
$$

For  any section $\varphi \in \Gamma(EM)$ we have
$
 \mathcal R (X \wedge Y)_\ast \, \varphi \;=\; R_{X,Y}\,\varphi
$.
It is easy to check that $q(R)$ acts as the Ricci tensor on tangent vectors. The definition of $q(R)$ is independent  of the 
orthonormal frame of  $\Lambda^2 \T M$, i.e. $q(R)$ can be written as 
$
q(R) = \sum \omega_{i \, \ast} \circ \mathcal R(\omega_i)_\ast
$
for any orthonormal frame of $\Lambda^2\T M$. Moreover it is easy to verify that $q(R)$ is a symmetric endomorphism of the vector bundle
$EM$.

The action of $q(R)$ on a symmetric $p$-tensors $K$ can be written as
$
q(R) K = \sum e_j \cdot    e_i \, \lrcorner\, R_{e_i\, e_j} K 
$.
On symmetric $2$-tensors  $h$ the curvature
endomorphism $q(R)$ is related to the classical curvature endomorphism  $\Rh$ (cf. \cite[p. 52]{besse}), which is defined by
$$
(\Rh h)(X, Y) \;=\; \sum h (R_{X, e_i} Y, e_i)\ .
$$
If $h$ is considered as a symmetric endomorphism the action of $\Rh$ on $h$ can be written as
$
\Rh (h) (X) = - \sum R_{X, e_i} \, h(e_i)\ 
$.

The action of  $\Ric$ is extended to  symmetric $2$-tensors $h$ as a derivation, i.e. it is defined as  
$
\Ric (h)(X, Y)  = - h(\Ric X, Y) - h(X, \Ric Y) \ 
$. 
Then the
following formula holds on $\Sym^2 \T M$:
\begin{equation}\label{qrh}
q(R) \;=\; 2\, \Rh \;-\; \Ric \ .
\end{equation}
If $h$ is the metric $g$ then $(\Rh g)(X, Y) = - \Ric(X,Y)$ and 
$\Ric (g)(X,Y) = - 2\, \Ric (X,Y)\ $.

As seen above, the covariant derivative $\nabla$  on $\Sym^p_0 \T M$ decomposes into three components defining
 three first order differential operators: $P_i(K): = p_i(\nabla K), \; i=1,2,3$, where $p_i$ are the orthogonal
 projections onto the three summands in the decomposition \eqref{deco}. The operators $P^*_i P^{\phantom{\ast}}_i, \, i=1,2,3$
 are then second order operators on sections of  $\Sym^p_0 \T M$. These three operators are linked
 by a Weitzenb\"ock formula:
 
 \begin{epr}\label{wbf}
 Let $K$ be any section of $\Sym^p_0 \,\T M$, then:
 $$
 q(R)\, K \; =\; -\, p \, P_1^*P^{\phantom{\ast}}_1 \, K \;+\; (n+p-2)\,P_2^*P^{\phantom{\ast}}_2 \, K \;+\; P_3^*P^{\phantom{\ast}}_3\, K \ .
  $$
 \end{epr}
\proof
The stated Weitzenb\"ock formula can be obtained as a special case of a general procedure 
described in \cite{gu}. However it is easy to check it directly using the following remarks.

Let $E$ be any $\SO(n)$-representation defining a vector bundle $EM$ and let $\T $ be the 
standard representation defining the tangent bundle $\T M$. Then any $p \in \End (\T \otimes E)$
can be interpreted as an element in $\Hom (\T \otimes \T \otimes E, E)$ defined as
$p(a\otimes b \otimes e) = (a \,  \lrcorner\otimes \, \id) \, p(b \otimes e)$, for $a, b \in \T $ and $e \in E$.
Important examples of such endomorphisms are the  orthogonal projections $p_i,\, i =1, \ldots, N$, onto the summands
in a decomposition $\T  \otimes E = V_1 \oplus \ldots \oplus V_N$.  Another example is the so-called
{\it conformal weight operator}  $B \in \Hom (\T \otimes \T \otimes E, E)$ defined as
$
B(a \otimes b \otimes e) = (a \wedge b)_* e
$.
As an element in $\End (\T \otimes E)$, the conformal weight operator can be written as
$
B(b \otimes e) = \sum e_i \otimes (e_i \wedge b)_\ast e \ 
$.

Let $K$ be a section of $EM$, then $\nabla^2K = \sum e_i \otimes e_j \otimes \nabla^2_{e_i, e_j}K$
is a section of the bundle $\Hom (\T M \otimes \T M \otimes EM, EM)$.
Using the remark above we can apply elements of the bundle $\End(\T M \otimes EM)$ to $\nabla^2 K$.
 It is then easy to check that
$$
B (\nabla^2 K ) \;=\; q(R)\, K,\qquad \id (\nabla^2 K) \;=\; -\, \nabla^*\nabla K, \qquad p_i(\nabla^2K) \;=\; - \,P^*_i P^{\phantom{\ast}}_i K
$$
where $P_i, \, i=1, \ldots, N$ are the first order differential operators $P_i(K) := p_i(\nabla K)$. 
Hence in order to prove the Weitzenb\"ock formula above it is enough to 
verify the following equation for endomorphisms of $\T  \otimes E$ in the case $E = \Sym^p \T $:
\bea
B &=& p\, p_1 \; - \; (n+p-2) \, p_2 \;-\, p_3
\;=\;
(p+1) \, p_1 \; - \; (n+p-3) \, p_2 \;-\, \id\\[1ex]
&=&
\pi^*_1\,\pi_1 \;-\; \tfrac{n+2p-4}{n+2p-2}\, \pi^*_2\,\pi_2 \;-\;\id \ .
\eea
This is an easy calculation using the explicit formulas for $\pi_i^*$ and $ \, \pi_i, \, i=1,2$ given above.
\qed


\subsection{Eigenvalue estimates for the Lichnerowicz Laplacian}

The Lichnerowicz Laplacian $\Delta_L$ is a Laplace-type operator acting on sections of $\Sym^p \T M$. It can be
defined by  $\Delta_L := \nabla^\ast \, \nabla \,+\, q(R)$. On symmetric $2$-tensors it is usually written as
$\Delta_L = \nabla^\ast \nabla \,+\, 2\, \Rh \;-\; \Ric $, which is the same formula, by \eqref{qrh}.

\begin{epr}\label{bound}
Let $(M^n, g)$ be a compact Riemannian manifold. Then $\Delta_L \ge 2\,q(R)$ holds on  the space of
divergence-free symmetric tensors. Equality  $\Delta_L h = 2 q(R) h$ holds for a divergence free tensor $h$ if and only if  $h$ is a Killing tensor.
\end{epr}
\proof
Directly from the definition we calculate
$$
\dd \d h = -\sum e_i \cdot \nabla_{e_i} (e_j \,\lrcorner\, \nabla_{e_j} h) = - \sum e_i \cdot (e_j \,\lrcorner\, \nabla_{e_i} \nabla_{e_j} h) \ .
$$
Similarly we have
\bea
\d \dd  h &=& - \sum e_i  \,\lrcorner\, \nabla_{e_i } (e_j \cdot \nabla_{e_j} h) \;=\;  - \sum e_i \,\lrcorner\, (e_j \cdot \nabla_{e_i} \nabla_{e_j} h) \\[1ex]
&=&
-\sum \nabla_{e_i} \nabla_{e_i}h\; -\; \sum e_j \cdot ( e_i\, \lrcorner\, \nabla_{e_i} \nabla_{e_j} h) \\ [1ex]
&=&
\nabla^\ast \nabla h\; -\; \sum e_i \cdot ( e_j\, \lrcorner\,   \nabla_{e_j} \nabla_{e_i} h)
\eea

Taking the difference we immediately  obtain 
$$
\d \dd  h \;-\; \dd  \d h \; =\;  \nabla^\ast \nabla h  \; -\; \sum e_i \cdot e_j\, \lrcorner\,   R_{e_j, e_i} h \;=\;
 \nabla^\ast \nabla h \;-\; q(R)h  \;=\; \Delta_L h - 2 q(R) h \ .
$$
Thus, if $h$ is divergence free we have $(\Delta_L - 2q(R)) h = \d \dd  h$ and the inequality follows after taking the $L^2$ product with $h$. The equality case is clearly characterized by $\dd h = 0$.
\qed

\begin{ere}
For symmetric $2$-tensors this estimate for $\Delta_L$ was proven in \cite{gibbons03}.
\end{ere}

\begin{ere}
As a consequence of Proposition \ref{bound}
we see that divergence free Killing tensors  on compact Riemannian manifolds are characterized by the equation 
$\nabla^\ast \nabla h = q(R) h $. This generalizes the well known characterization of Killing vector fields as divergence free
vector fields $\xi$ with $\nabla^\ast \nabla \xi = \Ric(\xi)$.
\end{ere}
\begin{ere}\label{eigenspace}
Recall that $q(R)$ is a symmetric endomorphism. The eigenvalues of $q(R)$ are constant on homogeneous spaces. On symmetric
spaces $M=G/K$ the Lichnerowicz Laplacian $\Delta_L$ can be identified with the Casimir operator $\Cas_G$ of the group $G$ and
$q(R)$ with the Casimir operator $\Cas_K$ of the group $K$. 
\end{ere}


\subsection{Non-existence results}

In \cite{dairbekov} Dairbekov and Sharafutdinov show the non-existence of trace-free conformal Killing tensors
on manifolds with negative  sectional curvature. In this section we will give a short new proof of this result.
 
\begin{epr} \label{ptf}
On a compact Riemannian manifold $(M, g)$ of non-positive sectional curvature
any trace-free conformal Killing tensor has to be parallel. If in addition there exists a point in $M$
where the sectional curvature of every two-plane is strictly negative, then $M$ does not 
carry any (non-identically zero) trace-free conformal Killing tensor.  
\end{epr}
\proof
On sections of $\Sym^p_0 \T M$ we consider  the Weitzenb\"ock formula of Proposition \ref{wbf}:
$$
q(R) \;=\; - p\, P_1^*P^{\phantom{\ast}}_1 \;+\;  (n+p-2)\,P_2^*P^{\phantom{\ast}}_2 \;+\; P_3^*P^{\phantom{\ast}}_3 \ .
$$ 
Trace-free conformal Killing tensors are characterized by the equation $P_1 K = 0$.
In particular we obtain for the $\mathrm L^2$-scalar product:
\be\label{estimate}
(q(R)K, K)_{\mathrm L^2} = (n+p-2)\|P_2K\|^2 + \|P_3K \|^2 \ge 0 \ ,
\ee
where $K$ is a  trace-free conformal Killing tensor.
We will show that $(q(R)K, K)_{\mathrm L^2} \le  0$ holds on a manifold with non-positive 
sectional curvature. This together with \eqref{estimate} immediately implies  that $P_2K = 0$ and $P_3K=0$ 
and thus that $K$ has to be parallel.

\medskip

For any $x \in M$ and any fixed tangent vector $X \in \T_xM$ we consider the symmetric 
bilinear form 
$
B_X(Y, Z) := g(R_{X, Y}X, Z)
$,
defined on tangent vectors $Y, Z \in \T_xM$. Since the sectional curvature is non-positive, this
bilinear form is positive semi-definite. Hence there is an orthonormal basis  $e_1, \ldots , e_n$  of $\T_xM$ (depending on $X$) with
$
B_X(e_i, e_j) =0
$
for $i\neq j$ and $B_X(e_i, e_i ) = a_i(X) \ge 0$ for all $i$.

A symmetric tensor $K\in \Gamma(\Sym^p_0TM)$ can also be considered as a polynomial map on $\T_M$ by the formula
$K(X) := g(K, X^p)$. In particular we have for the Riemannian curvature
$
(R_{Y,Z} K)(X) = (\sum R_{Y,Z}e_k \cdot e_k\, \lrcorner \, K)(X) = p  \sum g(R_{Y,Z}e_k, X) \, (e_k \, \lrcorner \, K)(X)
$.

Let $T$ be the tangent space $T=\T_xM$ for some $x\in M$. Then we can define a scalar product
on $\Sym^p_0T$ by
$
\tilde g (K_1, K_2) : = \int_{S_T}K_1(X) K_2(X) \dd\mu
$,
where $S_T$ is the unit sphere in $T$ and $\dd \mu$ denotes the standard Lebesgue measure on $S_T$.
From Schur's Lemma it follows the existence of a non-zero constant $c$ such that
$
\tilde g (K_1, K_2) = \frac1c \,g(K_1, K_2)
$
holds for all $K_1, K_2 \in \Sym^p_0T$. Since both scalar products are positive definite the
constant $c$ has to be positive. 
We now compute the scalar product  $g(q(R)K,K)$ at some point $x\in M$. From the remarks above we obtain:
\bea
g(q(R)K, K) &=& \sum  g(e_j \cdot e_i \lr R_{e_i, e_j}K,  K) \;=\; \sum g(R_{e_i, e_j}K, e_i \cdot e_j \,\lrcorner\, K)\\
&=&c \int_{S_T}\sum\,  (R_{e_i, e_j}K)(X)\cdot (e_i \cdot e_j \,\lrcorner\, K)(X) \, \dd\mu\\
&=&
c \,p^2  \int_{S_T}\sum \, g(R_{e_i, e_j}e_k,X) \cdot (e_k \lr K)(X) \cdot g(e_i, X) \cdot (e_j \lr K)(X) \, \dd\mu
\\
&=&
-\, c\,p^2  \int_{S_T}\sum \, g(R_{X, e_j}X, e_k) \cdot (e_k \lr K)(X)  \cdot (e_j \lr K)(X) \, \dd\mu\\
&=&
-\, c \,p^2 \int_{S_T} \sum \, B_X(e_j, e_k)  \cdot (e_k \lr K)(X)  \cdot (e_j \lr K)(X) \, \dd\mu\\
&=&
-\, c \,p^2 \int_{S_T} \sum \, a_j(X) \, ((e_j \lr K)(X) )^2 \, \dd\mu\\
&\le& 0\ .
\eea
This proves that for every trace-free symmetric tensor $K$, on a manifold with non-positive sectional curvature, the inequality $g(q(R)K, K) \le0$ holds at every point. By the above arguments, if $K$ is conformal Killing, then $K$ has to be parallel. 

\medskip

If in addition there is a point $x \in M$ where all sectional
curvatures are negative, then the symmetric form $B_X$ is positive definite for all $X\in S_T$, so its eigenvalues are positive: $a_j(X)>0$. The computation above shows that $(Y \lr K)(X) =0$ for every $X\in S_T$ and for all tangent vectors $Y$ orthogonal to $X$. This is equivalent to 
\beq\label{kxy}0=g(K,(|X|^2Y-\langle X,Y\rangle X)\cdot X^{p-1})=|X|^2g(K,Y\cdot X^{p-1})-\langle X,Y\rangle g(K,X^p)\eeq
for all tangent vectors $X,Y\in \T_xM$. On the other hand, from \eqref{commu2} we immediately get for every $X,Y\in \T_xM$
$$g(Y\cdot K,X^{p+1})=(p+1)\langle X,Y\rangle g(K,X^p)$$
and 
$$g(\LL\cdot(Y\lrcorner K),X^{p+1})=p(p+1)|X|^2g(K,Y\cdot X^{p-1})\ .$$
From \eqref{kxy} we thus obtain
$$\LL\cdot(Y\lrcorner K)=p\,Y\cdot K,\qquad \forall\ Y\in \T_xM\ .$$ 
Applying $\L$ and using \eqref{commu} and \eqref{commu2} together with the fact that $\L K=0$, yields $$(2n+4(p-1))Y\lrcorner K=2p\,Y\lrcorner K$$ at $x$, whence $K_x=0$. As $K$ is parallel from the first part of the proof, this shows that $K\equiv 0$.
\qed

\medskip

\begin{ecor}\label{cor}
Let $\Sigma_g$ be a compact Riemannian surface of genus $g\ge 2$. Then $M$ admits no trace-free
conformal Killing tensors. More generally, there are no trace-free conformal Killing tensors on compact quotients
of symmetric spaces of non-compact type.
\end{ecor}

\begin{ere}
Note that this result was also obtained by D.J.F. Fox in \cite{fox2}, Corollary 3.1.
\end{ere}


\section{Killing tensors with two eigenvalues}

Let $K \in \Gamma(\Sym^2_0 \T M)$ be a non-trivial trace-free Killing (i.e. St\"ackel) tensor on a connected Riemannian manifold $(M^n,g)$. We assume throughout this section that $K$ has at most  two eigenvalues at every point of $M$. 

The following result was proved by W. Jelonek (cf.  \cite[Theorem 2.1]{jelonek99}):

\begin{elem}
The multiplicities of the eigenvalues of $K$ are constant on $M$, so the eigenspaces of $K$ define two distributions
$\T M = E_1 \oplus E_2$. If $n_1, n_2$ denote the dimensions of $E_1, E_2$ and $\pi_i$ denote the orthogonal projection onto $E_i$ for $i=1,2$, then $K$ is a constant multiple of $n_2\pi_1-n_1\pi_2$.
\end{elem}
\proof
Since $K$ is trace-free, the eigenvalues of $K$ are distinct at every point $p\in M$ where $K_p\ne 0$. Every such point $p$ has a neighborhood $U$ on which the multiplicities of the eigenvalues of $K$ are constant. The eigenspaces of $K$ define two orthogonal distributions $E_1$ and $E_2$ along $U$ such that $\T M|_U=E_1\oplus E_2$. 
Then the Killing tensor $K$ can be written as
$
K = f\pi_1 + h\pi_2
$.   Since $K$ is trace-free, we have
$
0 = n_1 f + n_2 h
$.
The covariant derivative of $K$ can be written as
\bea
g( (\nabla_XK)Y, Z)
&=&
g( \nabla_X(KY) - K(\nabla_XY), Z)   \\[1ex]
&=&
g( X(f) \pi_1(Y) + f (\nabla_X\pi_1)Y + X(h) \pi_2(Y) + h(\nabla_X\pi_2)Y, Z  )\ .
\eea
Note that for any vector  $X$ and vector fields $X_i, Y_i \in E_i$ for $i=1,2$ we have
\begin{equation}\label{zero}
g( (\nabla_X\pi_1) X_i, Y_i) \;=\; 0 \ 
\end{equation}
and similarly for $\pi_2$.
For  $X\in E_1$, the Killing tensor equation gives $(\nabla_XK)(X,X) = 0$ and it follows
from the formula above that
$
g( X(f) \pi_1(X), X) = 0
$.
Thus $X(f)=0$ for all $X \in E_1$ and similarly $X(h)=0$ for all $X\in E_2$. It follows
that $f$ and $h$ are constant on $U$, since  $f$ and $h$ are related via $n_1 f + n_2 h = 0$. The eigenvalues of $K$ are thus constant on $U$. Since this is true on some neighbourhood of every point $p$ where $K_p\ne 0$, we deduce that the eigenvalues of $K$, and their multiplicities, are constant on $M$. This proves the lemma.
\qed

We will now characterize orthogonal splittings of the tangent bundle which lead to trace-free Killing tensors.

\begin{epr} \label{distributions}
Let $E_1$ and $E_2$ be orthogonal complementary distributions on $M$ of dimensions $n_1$ and $n_2$ respectively. Then the trace-free symmetric tensor $K = n_2\pi_1-n_1\pi_2$ is Killing if and only
if the following conditions hold:
 \begin{equation}\label{d1}
\nabla_{X_1}X_1 \in \Gamma(E_1) \qquad \forall \; X_1\in \Gamma(E_1)
\qquad \mbox{and} \qquad
\nabla_{X_2}X_2 \in \Gamma(E_2) \qquad \forall \; X_2\in \Gamma(E_2) \ .
\end{equation}
\end{epr}
\proof
Assume first that $K =n_2\pi_1-n_1\pi_2$ is a Killing tensor. Since $\pi_1 + \pi_2 = \Id$ is parallel, we see
that $\pi_1$ and $\pi_2$ are Killing tensors too. Let $X_1 \in \Gamma(E_1)$ and $X_2 \in \Gamma(E_2)$. As
$\pi_1$ is a Killing tensor, we get from (\ref{zero}):
\begin{eqnarray*}
0 &=& g( (\nabla_{X_1}\pi_1) X_1, X_2) + g( (\nabla_{X_1}\pi_1) X_2, X_1)+g( (\nabla_{X_2}\pi_1) X_1, X_1)\\&=&2g( (\nabla_{X_1}\pi_1) X_1, X_2)\\
&=&2g( \nabla_{X_1} X_1, X_2)\ , 
\end{eqnarray*}
and similarly $0=g( \nabla_{X_2} X_2, X_1)$, thus proving \eqref{d1}. 

Conversely, if \eqref{d1} holds, then for every vector field $X$ on $M$ we can write $X=X_1+X_2$, where $X_i:=\pi_i(X)$ for $i=1,2$ and compute using  (\ref{zero}) again:
\begin{eqnarray*}g((\nabla_X\pi_1)X,X)&=&2g((\nabla_X\pi_1)X_1,X_2)=2g(\nabla_XX_1,X_2)\\
&=&2g(\nabla_{X_1}X_1,X_2)+2g(\nabla_{X_2}X_1,X_2)\\
&=&2g(\nabla_{X_1}X_1,X_2)-2g(\nabla_{X_2}X_2,X_1)=0\ .
\end{eqnarray*}
\qed

Pairs of 
distributions with this property were studied in \cite{naveira} by A. Naveira under the name of almost
product structures of type $D_1$.
Note that Killing tensors with two eigenvalues were intensively studied by W. Jelonek in \cite{jelonek99},
\cite{jelonek09}, \cite{jelonek13}
and also by B. Coll et al.\ in \cite{coll}. In particular W. Jelonek proves that a Killing tensor with constant eigenvalues
satisfies our condition \eqref{d1}. Thus our Proposition \ref{distributions} is in some sense a converse to his result.

\begin{exe}
If $M\to N$ is a Riemannian submersion with totally geodesic fibers and $\mathcal{V}$ and $\mathcal{H}$ denote its vertical and horizontal distributions, then $E_1:=\mathcal{V}$ and $E_2:=\mathcal{H}$ satisfy \eqref{d1}, by the O'Neill formulas. It turns out that this generalizes several examples of Killing tensors appearing in the physics literature, e.g. in \cite{gibbons03}.
\end{exe}

\begin{ere}
Note that \eqref{d1} does not imply the integrability of $E_1$ or $E_2$. However, assuming that \eqref{d1} holds and that one of the distributions, say $E_1$, is integrable, then there exists locally a Riemannian submersion with totally geodesic fibers whose vertical and horizontal distributions are $E_1$ and $E_2$ respectively.
\end{ere}


\section{Conformal Killing tensors on hypersurfaces}

In this last section we give a short proof, using the formalism developed above, of a vanishing result of  Dairbekov and Sharafutdinov:

\begin{ath}[\cite{dairbekov}]\label{nodal}
Let $(M^n, g)$ be a connected Riemannian manifold and let $H \subset M$ be a hypersurface. If a trace-free conformal
Killing tensor $K$ vanishes along $H$, then $K$ vanishes identically on $M$.
\end{ath}
\proof
Let $K \in \Gamma(\Sym^p_0 \T M) $ be a trace-free conformal Killing tensor vanishing along a 
hypersurface $H \subset M$.
Starting with $K_0:=K$ we recursively define tensors $K_{l} := \nabla K_{l-1}$, which are
sections of $\T M^{\otimes \, l} \otimes \Sym^p_0\T M$. We claim that all tensors $K_i$
vanish along $H$. Since the conformal Killing equation is of finite type this will imply that $K$
is identically zero on $M$. 

Consider the natural extensions $\dd:\Gamma(\T M^{\otimes \, l} \otimes \Sym^p\T M)\to \Gamma(\T M^{\otimes \, l} \otimes \Sym^{p+1}\T M)$, $\delta:\Gamma(\T M^{\otimes \, l} \otimes \Sym^p\T M)\to \Gamma(\T M^{\otimes \, l} \otimes \Sym^{p-1}\T M)$ and $\nabla:\Gamma(\T M^{\otimes \, l} \otimes \Sym^p\T M)\to \Gamma(\T M^{\otimes \, (l+1)} \otimes \Sym^{p}\T M)$ of $\dd,\ \d$ and $\nabla$, defined on decomposable tensors by
$$\dd(T\otimes K):=\sum_ie_i\cdot\nabla_{e_i}(T\otimes K)=T\otimes \dd K+\sum_i\nabla_{e_i}T\otimes e_i\cdot K\ ,$$
$$\delta(T\otimes K):=-\sum_ie_i\lrcorner\nabla_{e_i}(T\otimes K)=T\otimes \delta K-\sum_i\nabla_{e_i}T\otimes e_i\lrcorner K\ ,$$
$$\nabla(T\otimes K):=\nabla T\otimes K+\sum_i (e_i\otimes T)\otimes \nabla_{e_i} K\ ,$$
where $\{e_i\}$ denotes as usual a local orthonormal basis of $\T M$. A straightforward computation shows that 
\begin{equation}\label{dnabla}
[\dd,\nabla]={\mathcal R}^+,\qquad [\delta,\nabla]=\mathcal{R}^-,
\end{equation}
where ${\mathcal R}^+:\T M^{\otimes \, l} \otimes \Sym^p\T M\to \T M^{\otimes \, (l+1)} \otimes \Sym^{p+1}\T M$ is defined by
$${\mathcal R}^+(T\otimes K):=\sum_{i,j}(e_i\otimes T)\otimes (e_j\cdot R_{e_j,e_i}K)+(e_i\otimes R_{e_j,e_i}T)\otimes (e_j\cdot K)\ ,$$
and ${\mathcal R}^-:\T M^{\otimes \, l} \otimes \Sym^p\T M\to \T M^{\otimes \, (l+1)} \otimes \Sym^{p-1}\T M$ is defined by
$${\mathcal R}^-(T\otimes K):=-\sum_{i,j}(e_i\otimes T)\otimes (e_j\lrcorner R_{e_j,e_i}K)+(e_i\otimes R_{e_j,e_i}T)\otimes (e_j\lrcorner K)\ .$$
Since $K$ is trace-free conformal Killing, Lemma \ref{conf} shows that 
\begin{equation}\label{ddk}\dd  K_0 \;=\; -\, \tfrac{1}{n+2p-2}\,\LL  \, \d K_0\ .
\end{equation}
(Note that since $K$ is trace-free, the notation $K_0$ from Lemma \ref{conf} coincides with our notation $K=K_0$ above).
We will prove by induction that there exist vector bundle morphisms $F_{i,l}:\T M^{\otimes \, i} \otimes \Sym^p\T M\to \T M^{\otimes \, l} \otimes \Sym^{p+1}\T M$ such that 
\begin{equation}\label{dkl}\dd  K_l \;=\; -\, \tfrac{1}{n+2p-2}\,\LL  \, \d K_l+\sum_{i=0}^{l-1}F_{i,l}(K_i)\ ,
\end{equation}
where here $\LL :\T M^{\otimes \, i} \otimes \Sym^p\T M\to \T M^{\otimes \, i} \otimes \Sym^{p+2}\T M$ denotes the natural extension of $\LL $, which of course commutes with $\nabla$.
For $l=0$, this is just \eqref{ddk}. Assuming that the relation holds for some $l\ge 0$ we get from \eqref{dnabla}
\bea\dd  K_{l+1} &=& \dd\nabla K_l=\nabla\dd K_l+\mathcal{R}^+K_l\\
&=&\nabla\left(-\tfrac{1}{n+2p-2}\,\LL  \, \d K_l+\sum_{i=0}^{l-1}F_{i,l}(K_i)\right)+\mathcal{R}^+K_l\\
&=&-\tfrac{1}{n+2p-2}\,\LL  \, \nabla\d K_l+\sum_{i=0}^{l-1}\left((\nabla F_{i,l})(K_i)+(\id\otimes  F_{i,l})(K_{i+1})\right)+\mathcal{R}^+K_l\\
&=&-\tfrac{1}{n+2p-2}\left(\,\LL  \,\d K_{l+1}+\LL  \,\mathcal{R}^-K_l\right)+\sum_{i=0}^{l-1}\left((\nabla F_{i,l})(K_i)+(\id\otimes  F_{i,l})(K_{i+1})\right)+\mathcal{R}^+K_l\ ,\eea
which is just \eqref{dkl} for $l$ replaced by $l+1$ and
$$F_{i,l+1}:=\begin{cases}\nabla F_{i,l}+(\id\otimes F_{i-1,l}),& i\le l-1\\ 
-\tfrac{1}{n+2p-2}\,\LL  \,\mathcal{R}^-+(\id\otimes F_{l-1,l})+\mathcal{R}^+,&i=l\ .
\end{cases}
$$
This proves \eqref{dkl} for all $l$. 

Assume now that $K_0, \ldots, K_{l}$ vanish along $H$ for some $l\ge 0$. We claim
that $K_{l+1}$ is also vanishing along $H$. Take any point $x\in H$ and choose a local orthonormal frame $\{e_i\}$ such that $e_1=:N$ is normal to $H$ and $e_2,\ldots,e_n$ are tangent to $H$ at $x$. From \eqref{dkl} we have $\dd  K_l \;=\; -\, \tfrac{1}{n+2p-2}\,\LL  \, \d K_l$ at $x$. Moreover, $\nabla_{e_i}K_l$ vanishes at $x$ for every $i\ge 2$. The previous relation thus reads 
\begin{equation}\label{nn}N\cdot\nabla_N K_l \;=\; \tfrac{1}{n+2p-2}\,\LL  \, N\lrcorner \nabla_N K_l\ .\end{equation}
Writing 
$$\nabla_N K_l=\sum_{I\in\{1,\ldots,n\}^l}e_I\otimes S_I\ ,$$
with $e_I:=e_{i_1}\otimes \ldots\otimes e_{i_l}$, and $S_I\in \Sym^p\T M$,
\eqref{nn} becomes 
\begin{equation}\label{nn1}N\cdot S_I=\tfrac{1}{n+2p-2}\,\LL  \, N\lrcorner S_I\end{equation}
for every $I$. 
It is easy to check that this implies $S_I=0$ for every $I$. Indeed, if $i_0\in\{0,\ldots,p\}$ denotes the largest index $i$ such that the coefficient $C_i$ of $N^i$ in $S_I$ is non-zero, comparing the coefficients of $N^{i_0+1}$ in \eqref{nn1} yields $C_{i_0}=\tfrac{1}{n+2p-2}i_0C_{i_0}$, which is clearly impossible for $n>2$. This shows that $\nabla_N K_l=0$ at $x$, and since we already noticed that $\nabla_{e_i}K_l$ vanishes at $x$ for every $i\ge 2$, we have $K_{l+1}=0$ at $x$. As this holds for every $x\in H$, our claim is proved. 

Consequently, if $K$ vanishes along $H$, then all covariant derivatives of $K$ vanish along $H$. Since the conformal Killing equation has finite type (cf. Remark \ref{ft}), this implies that $K$ vanishes identically on $M$ (being a component of a parallel section of some vector bundle on $M$ which vanishes along $H$). This proves the theorem.
\qed

\end{document}